\documentclass{amsart}[11pt] 

\usepackage{amsmath,amssymb,mathrsfs,amsthm,delarray}
\usepackage{pictex}
\usepackage[all,cmtip]{xy}
\usepackage{graphicx}
\numberwithin{equation}{section}

\newtheorem{Theorem}{Theorem}[section]
\newtheorem{Lemma}{Lemma}[section]
\newtheorem{definition}{Definition}[section]
\newtheorem{Coro}{Corollary}[section]
\newtheorem{prop}{Proposition}[section]

\newtheorem{Fact}{Fact}
\newtheorem{Rm}{Remark}[section]
\newcommand{\I}{\mathbb I}
\newcommand{\R}{\mathbb R}
\newcommand{\C}{\mathbb C}
\newcommand{\D}{\mathbb D}
\newcommand{\T}{\mathbb T}
\newcommand{\E}{\mathcal E}
\renewcommand{\L}{\mathcal L}

\newtheorem*{TA}{Theorem A}
\newtheorem*{TB}{Theorem B}
\newtheorem*{TC}{Theorem C}
\newtheorem*{keylemma}{Key Lemma}

\title[The full renormalization horseshoe for multimodal maps] {The full renormalization horseshoe for multimodal maps:the continuity of the anti-renormalization operator for multimodal maps}
\author{Yimin Wang}
\date{\today}
\address{Shanghai Center for Mathematical Sciences,
Fudan University, No. 2005 Songhu Road, Shanghai 200438, China}
\email{yiminwang16@fudan.edu.cn}

\begin{document}
\maketitle
\begin{abstract}
   In this paper, we consider the renormalization operator $\mathcal R$ for multimodal maps. We prove the renormalization operator $R$ is a self-homeomorphism on any totally $\mathcal R$-invariant set. As a corollary, we prove the existence of the full renormalization horseshoe for multimodal maps.
\end{abstract}
\bigskip
\section{Introduction}
Renormalization has been an important idea and tool in dynamical systems. Feigenbaum's renormalization conjecture says that a certain renormalization operator has a hyperbolic fixed point. In fact, the original case considered by Feigenbaum is the periodic doubling case \cite{feigenbaum1979universal}. And such a conjecture was also formulated by Coullet and Tresser independently from Feigenbaum. For the periodic doubling case, Lanford \cite{lanford1980remarks} proved the existence of the hyperbolic  fixed point with computer assistant ,  Sullivan  \cite{sullivan1988bounds}  and McMullen \cite{mcmullen1998rigidity} proved the uniqueness of the fixed point and the exponential contraction of $\mathcal R$. Finally, Lyubich \cite{lyubich1999feigenbaum} considered the renormalization operator $\mathcal R$ on the space $\mathcal{QG}$ of quadratic-like germs, he defined a complex structrue on $\mathcal{QG}$ and then proved the hyperbolicity of the renormalization horseshoe. In \cite{regularorstochastic},  Lyubich proved the set of infinitely renormalizable real polynomials has Lebesgue measure zero. It implies his famous result: a typical real polynomial is either regular or stochasitc. Avila and Lyubich\cite{Fullhorseshoeunimodal} generalize the result to analytic unimodal case by introducing a method of path holomorphic structure and cocycles. There are also parallel results about the renormalization conjecture for critical circle maps, see \cite{yampolsky1999complex,yampolsky2003hyperbolicity,yampolsky2003renormalization}.\par
In  \cite{smania2001renormalization}, Smania introduced  {multimodal maps of type $\bf N$} and proved that deep renormalizations of infinitely renormalizable multimodal maps are multimodal maps of type $\bf N$ for some positive integer ${\bf N}$. Let $\I=[-1,1]$.  A multimodal map $f:\I\to \I$ is called {\em a multimodal map of type $\bf N$}, if there exists unimodal maps $f_0,\cdots,f_{{\bf N}-1}$ with following properties:
\begin{enumerate}
\item $f_j:\I\to \I$ is a unimodal map fixing $-1$;
\item $f=f_{{\bf N}-1}\circ\cdots\circ f_0$;
\item $0$ is a quadratic critical point of $f_j$ such that $f_j(0)\ge 0$ and $f''_j(0)<0$.
\end{enumerate}
We will call $(f_0,f_1,\cdots,f_{{\bf N}-1})$ {\em a unimodal decomposition} of $f$. For convenience, we will also assume that $f$ is even, i.e. $f(x)=f(-x)$ for all $x\in \I$. Since we concern about the infinitely renormalizable case, such an assumption will not lose generality.\par
A multimodal map $f$ of type $\bf N$ is called {\em renormalizable} if there exists a periodic interval $J$ of  period $p$ of $f$ such that $f^p|_J$ is affinely conjugate to a multimodal map of type $\bf N$. There is a canonical way to normalize $f^p|_J$ to be a multimodal map of type $\bf N$, and we call the normalized map $\mathcal R f$ the renormalization of $f$ and the smallest integer $p$ is called {\em the renormalization period} of $f$.  We say $f$ is {\em infinitely renormalizable with bounded combinatorics} if $f$ is infinitely renormalizable and the renormalization period $p_k$ of $\mathcal R^k f$ is bounded.\par
Let $\mathcal I$ be the set of all the infinitely renormalizable real-analytic multimodal maps of type $\bf N$ equipped with the $C^3$-topology. Then the renormalization operator $\mathcal R$ for multimodal maps of  type $\bf N$ induces a dynamical system $\mathcal R:\mathcal I\to \mathcal I$.
In \cite{SmaniaPhase,smania2016solenoidal}, Smania  considered the sub-dynamical sysetem : $\mathcal R|_{\mathcal I_p}:\mathcal I_p\to \mathcal I_p$, where $\mathcal I_p$ is the set of infinitely renormalizable multimodal maps of type $\bf N$ with combinatorics bounded by $p$. He proved that the $\omega$-limit set $\Omega_p$ of the renormalization operator $\mathcal R|_{\mathcal I_p}$ is compact and $\mathcal R|_{\Omega_p}:\Omega_p\to \Omega_p$ is topologically conjugate to a full shift of finite elements. 

In this paper, we prove the renormalization operator of multimodal maps of type $\bf N$ has a full horseshoe:
\begin{TA}Let $\mathcal I$ be the set of all the infinitely renormalizable multimodal maps of type $\bf N$  and $\Sigma$ be the set of all the renormalization combinatorics for multimodal maps of type $\bf N$. Then there exists a precompact subset $\mathcal A\subset \mathcal I$ such that the restriction $\mathcal R|_{\mathcal A}$ of $\mathcal R$ is topologically conjugate to a two-sided  full shift  on $\Sigma^{\mathbb Z}$.
\end{TA}
See section~3 for a definition of the renormalization combinatorics.\par
To prove the full renormalization horseshoe for multimodal maps of type $\bf N$, there are two main difficulties. One is to prove infinitely renormalizable multimodal maps has complex bounds, which has been done by Shen\cite{C2Density}. Since such  complex bounds has been built, we can modify  the argument of Avila-Lyubich\cite{Fullhorseshoeunimodal} to get a semi-conjugacy desired in Thereom~A:
\begin{TB}Let $\mathcal I$ and $\Sigma$ be as in the assumptions of Theorem~A. Then there exists a precompact subset $\mathcal A\subset \mathcal I$ such that $\mathcal R(\mathcal A)=\mathcal A$ and a continuous bijection $h$ which gives a topological semi-conjugacy between $\mathcal R|_{\mathcal A}$ and  a two-sided  full shift  on $\Sigma^{\mathbb Z}$.
\end{TB}
Another difficuliy is to prove the inverse $h^{-1}$ of the semi-conjugacy in Theorem~B is continuous. For the uniformly bounded combinatorics case, the proof is easy. However, it is not trivial to deal with the unbounded combinatorics case. To this end, we prove the following dichotomy:

\begin{keylemma}Let $\{f_k\}$ be a sequence of bi-infinitely renormalizable multimodal maps of type $\bf N$ which is precompact under $C^3$-topology. If the renormalization periods $p_k$ of $f_k$ tend to infinity, then each limit of $\mathcal Rf_k$ is either a polynomial of degree $2^n$ or with bounded real trace. 
\end{keylemma}
It is worth mentioning that Avila-Lyubich proved this theorem in the unimodal case\cite{Fullhorseshoeunimodal}, which inspired us so much. Even in that case, the proof is nontrivial and complicated.\par
As a corollary of the Key Lemma, we prove
\begin{TC} For any totally $R$-invariant precompact subset $\mathcal A' \subset \mathcal I$, the restriction $\mathcal R^{-1}|_{\mathcal A'}$ of the anti-renormalization operator $\mathcal R^{-1}$ to $\mathcal A'$ is continuous.
\end{TC}
Let us now describe the organization of the paper.\par
The proof of Theorem~A will be postponed to section~5.  In section~2, we recall some background of real box maps and use the distortion results for real box maps to prove some compactness lemmas which are the indispensable tools in the proof of the Key Lemma. We recall the definition of renormalization combinatorics in section~3 and then prove Theorem~C. From section~4 to the end of this paper, we will use the idea of path holomorphic space together with Theorem~C, following Avila-Lyubich, to show the existence of the full renormalization horseshoe for multimodal maps of type $\bf N$. We study the complexification of the multimodal maps of type $\bf N$ and its renormalization operator  in section~4, the external and inner structure will be discussed there. The  path holomorphic structure on each hybrid leaf will be defined in section~5 and we modify the argument of Avila and Lyubich \cite{Fullhorseshoeunimodal} to show that the renormalization operator contracts exponentially fast along the real-symmetric hybrid leafs by virtue of the complex bounds.

\subsection*{Acknowledgement}
The author would like to thank his supervisor  Weixiao Shen for advice and helpful discussions on this problem.

\section{Renormalization Operator and infinitely renormalizable maps}
In this section, we will first introduce the definition of renormalization of multimodal maps of type $\bf N$ and recall some results about the real bounds for real box maps. Then we will prove the Main Theorem:
\begin{keylemma}Let $\{f_k\}$ be a sequence of bi-infinitely renormalizable multimodal maps of type $\bf N$ which is precompact under $C^3$-topology. If the renormalization periods $p_k$ of $f_k$ tend to infinity, then each limit of $\mathcal Rf_k$ is either a polynomial of degree $2^n$ or with bounded real trace. 
\end{keylemma}
\subsection{The extended maps and renormalization}
Fix a multimodal map $f$ of type $\bf N$ and a unimodal decomposition $(f_0,f_1,\cdots,f_{{\bf N}-1})$ of $f$. Let $\I_{\bf N}=\{(x,j)\mid x\in \I,~0\le j<{\bf N}\}$,  following Smania \cite{SmaniaPhase}, we define the extended map $F$ of $f$:
\begin{align*}
F:\Bbb I_{\bf N} & \longrightarrow  \I_{\bf N}\\
(x,j) & \longmapsto (f_j(x), j+1 \mod {\bf N}).
\end{align*}
Clearly, the extended map of $f$ is not unique since $f$ can have several unimodal decompositions. The extended map is a real box map.
\begin{definition}\label{def of re}A closed interval $J \ni 0$ is called a $k$-periodic interval of an extended map $F$ if it satisfies:
\begin{enumerate}
\item $F^{k}(J\times \{0\}) \subset J\times\{0\}$,
\item $J\times \{0\}, F(J\times \{0\}),\cdots, F^{k-1}(J\times \{0\})$ are closed intervals with disjoint interiors,
\item for every $1\le j\le {\bf N}-1$, there exists exactly one $1\le m<k$ such that $0\times\{j\} \in F^m(J\times \{0\})$,
\item $k>{\bf N}$.
\end{enumerate} 
Let $p=k/{\bf N}$, we also say $J$ is a $p$-periodic interval of $f$. 
\end{definition}
Let $F$ be an extended map of a multimodal map $f$ of type $\bf N$. Consider a maximal $k$-periodic interval $J$ of $F$, i.e., $F^{k}(\partial J\times\{0\})\subset \partial J\times\{0\}$ (If $F$ has a $k$-periodic interval, then it must have a maximal $k$-periodic interval). Then there exists a canonical affine transformation $A_0: J \to \I$   such that $A_0(0)=0$ and $A_0\circ f^{k/{\bf N}}\circ A_0^{-1}:\I \to \I$ is a multimodal map of type $\bf N$. To see this, for every $0\le j <{\bf N}$, let $m_j$ be the integer such that $0\times \{j\} \in F^{m_j}(J\times \{0\})$ and $J_j$ be the interval such that $J_j\times\{j\}$ is the symmetrization of $F^{m_j}(J\times \{0\})$ with respect to $(0,j)$. There is a periodic point $z_j$ of $f$ on the boundary of $J_j$, let $A_j:J_j\to \I$ be the affine transformation such that $A_j(0)=0$, $A_j(z_j)=-1$ and $\tilde A_j: J_j\times\{j\} \to \I$, $\tilde A_j(x,j)=(A_j(x),j)$ for all $0\le j<{\bf N}$. For convenience, we make a convention that $m_n=m_0=0$, $J_{\bf N}=J_0$, $A_{\bf N}=A_0$, $\tilde A_{\bf N}=\tilde A_0$. Then for every $0\le j\le {\bf N}-1$, $F_j:=\tilde A_{j+1}\circ F^{m_{j+1}-m_j}\circ \tilde A^{-1}_j:\I\times \{j\} \to \I\times \{j+1 \mod {\bf N}\} $ is a unimodal map with critical point $0$ and $\tilde A_0\circ F^k\circ \tilde A^{-1}_0=F_{{\bf N}-1}\circ\cdots\circ F_1\circ F_0$, it implies $A_0\circ f^{k/{\bf N}}\circ A_0^{-1}:\I\to \I$ is also a multimodal map of type $\bf N$. If $f$ does not have a periodic interval with period strictly smaller that $k/{\bf N}$, then we call $A_0\circ f^{k/{\bf N}}\circ A_0^{-1}$ is the {\it real renormalization} of $f$ and denote it by $\mathcal R f$. Clearly, the renormalization of $f$ does not depend on the unimodal decomposition of $f$.  If $\mathcal R f$ is again renormalizable, then we will say $f$ is {\em  twice renormalizable}. If this procedure can be done infinitely many times, then $f$ will be called {\it infinitely renormalizable}. In this paper, we mainly concern about the infinitely renormalizable maps. \par
\begin{definition}A multimodal map $f$ of type $\bf N$ is called anti-renormalizable, if there exists a renormalizable multimodal map $g$ of type $\bf N$ such that $\mathcal Rg=f$. 
\end{definition}
Since Smania had proved the renormalization operator $\mathcal R$ is an injection  \cite [Proposition 2.2]{smania2016solenoidal}, the anti-renormalization operator is also well-defined. Similar to infinitely renormalizable maps, we can define infinitely anti-renormalizable maps, and a multimodal map is called {\it bi-infinitely renormalizable} if it is both infinitely renormalizable and anti-infinitely renormalizable.\par

\subsection{Background in real box maps}
Throughout Section~2, we will assume $f$ is a bi-infinitely renormalizable multimodal map of type $\bf N$ with renormalization period $p>2$ and fix an extended map $F$ of $f$. Set $ {\bf c_j}=(0,j)\in \I\times\{j\}$ for all $0\le j\le {\bf N}-1$. Assume $(\alpha,0)$ is the $F^n$- fixed point closest to $\bf c_0$ and $(-\alpha,0)$ is the reflection of $(\alpha,0)$ with respect to $\bf c_0$. Set
\[I_0:=(\alpha,-\alpha)\times\{0\}.\]
Let $(\beta,0)$ be the preimage of $F^{-n}(\alpha,0)$ closest to the point $(-1,0)$ and define a set 
$$\mathfrak E_0:=\I\times \{0\} \setminus\{(\beta,0),(-\beta,0),(\alpha,0),(-\alpha,0)\}.$$
\begin{definition}An open subset $B$ of $\mathbb I_{\bf N}$ is called nice if $\bigcup\limits_{k\ge 1}F^k(\partial B) \cap B=\emptyset$.
\end{definition}
The concept of nice interval was first introduced by Martens \cite{martens1994distortion}. For a nice symmetric interval $B$, we denote $D_B$ the first entry domain of $B$ under the iterates of $F$,  that is, 
\[D_B=\{x\in \I_{\bf N}\mid F^k(x) \in B~\text{for  some  integer}~k\ge1\}.\] 
For any $x\in D_B$, the minimal positive integer $k=k(x)$ such that $F^k(x)\in B$ is called  {\it the first entry time} of $x$. The {\it first entry map} to $B$ is defined as:
\begin{align*}
R_B:D_B & \longrightarrow  B\\
x & \longmapsto F^{k(x)}(x).
\end{align*}
The restriction of $R_B$ to $D_B\cap B$ is called {\it the first return map}. For $x\in D_B$, we shall use $\mathcal L_x(B)$ to denote the connected component of $D_B$ containing $x$. Let $\mathcal L^1_x(B)=\mathcal L_x(B)$, and for any positive integer $j\ge 2$, let $\mathcal L^{j}_{x}(B)=\mathcal L_{x}(\mathcal L^{j-1}_x(B))$, whenever it makes sense.\par
For a symmetric nice interval ${\bf c}\in I \subset \I_{\bf N}$,  the {\it scaling factor} of $I$ is defined as: 
\[\lambda_I:=\frac{|I|}{|\L_{\bf c}(I)|}.\]

For a nice open set $K\cap \omega({\bf c})\ne \emptyset$ (where $ \omega({\bf c})$ is the $\omega$-limit set of $\bf c$ ), let $\mathcal M(K)$ be the collection of intervals which are pullbacks of components of $K$. Shen\cite{C2Density} defined {\it the limit scaling factor} of $K$ as:
\[\Lambda_K:=\sup\limits_I \lambda_I,\]
where the supremum is taken over all symmetric nice intervals in $\mathcal M(K)$.\par
We say $B_0\Supset B_1\Supset B_2\Supset\cdots$ is a {\em nest} if  there exists $x\in B_0$ such that $B_{n+1}=\mathcal L_x(B_n)$ for all $n\in \mathbb N$.\par
A sequence of nice symmetric intervals $B_1\Supset B_2\Supset\cdots\Supset B_L\ni {\bf c}$ is called {\it a central cascade} with respect to $\bf c$, if $B_{j+1}=\mathcal L_{\bf c}(B_j)$ for all $1\le j\le L-1$ and $R_{B_j}({\bf c})\in B_{j+1}$ for each $1\le j\le L-2$.  Such a central cascade is called maximal if $R_{B_{L-1}}$ displays a non-central return, i.e., $R_{B_{L-1}} ({\bf c})\notin B_L$.  We say that the central cascade is of {\it saddle node} type
if $R_{B_1}|_{B_2}$ has all the critical points in $B_L$,  and does not have a fixed point.
\begin{Rm}If $B_1\Supset B_2\Supset\cdots\Supset B_L$ is a  maximal central cascade, then $B_j\Supset B_{j+1}\Supset\cdots\Supset B_L$ is also a maximal central cascade for all $2\le j\le L-1$.
\end{Rm}
By  a {\em chain} we mean a sequence of  open intervals $\{G_s\}^k_{s=0}$ such that  $G_{s+1}$ is a component of $F^{-1}(G_s)$  for every $0\le s\le k-1$.  The order of the chain is the number of the integers $s$ with $0\le s< n$ such that $G_s$ intersects $\mathrm{Crit}(F)$ and the intersection multiplicity is the maximal number of the intervals $G_s$ ($0 \le s \le n$)  which have a non-empty intersection.\par

We shall need the following known results:
 \begin{Theorem}[{\cite[Theorem A]{SE}}]\label{realbounds}There exists $1<\lambda = \lambda(\|F\|_{C^3})$ with the following property. Let us consider a nest $B_0\Supset B_1\Supset B_2\Supset\cdots$.  If $R_{B_k}$
does not display a central return, then
\[\lambda_{B_{k+1}} =\frac{|B_{k+1}|}{|B_{k+2}|}>\lambda.\]
\end{Theorem}
Real bounds for $S$-unimodal maps were proved earlier by Martens \cite{martens1994distortion}.\par
We say an open interval $I$ is a {\em $\delta$-neighborhood} of an interval $J$, which is denoted by $(1+2\delta)J$, if $J\Subset I$ and each component $I \setminus J$ has length equal to $\delta|J|$.
\begin{Lemma}[{\cite[Proposition~2.2]{Shen1}}]\label{distortion}For any $p,q\in \mathbb N$ and any $\delta> 0$, there
exists a constant $\delta_1 = \delta_1(\delta,p,\|F\|_{C^3}) > 0$ such that the
following holds. Let $G = \{G_j \}^s_{j=0}$ and $G'= \{G'_j\}^s_{j=0}$ be chains such that 
$G_j\Subset G'_j$ for all $0\le j\le s$. Assume the order of $G'$ is at most $p$ and 
\[\#\{j\mid G'_j \Supset G_s\} \le q.\]
If $(1+2\delta)G_s \subset G'_s$, then 
$(1+2\delta_1)G_0 \subset G'_0$. Moreover, $\delta_1\to \infty$ as $\delta \to \infty$.
\end{Lemma}
For two intervals $I$ and $J$, we briefly say $J$ is geometrically deep inside $I$ if there is a large $\delta$ such that $(1+2\delta)J\subset I$.

\begin{Rm}If $F^q|_{I}:I\to I'$ is a first return map and $J'$ is geometrically deep inside $I'$, then it follows immediately from Lemma~\ref{distortion} that $J:=(F^q|_{I})^{-1}(J')$ is geometrically deep inside $I$.
\end{Rm}

All the central cascades have been  proved to be  essentially saddle-node in the following sense by Shen\cite{C2Density}:

\begin{Lemma}[{\cite[Proposition~5.1, Theorem~5.4]{C2Density}}]\label{sn}
For any $\delta>0$ and $\rho>0$, there exists $b=b(\delta, \rho,\|F\|_{C^3})$ and $\eta=\eta(\delta, \rho,\|F\|_{C^3})>1$ with the following property. Consider a central cascade $B_1\Supset B_2\Supset \cdots\Supset B_L$, assume $B_1\supset (1+2\delta)B_2$, $\Lambda_{B_1}<\rho$ and $L>3b$. Then
\begin{enumerate}
\item the central cascade $B_{b}\Supset B_{b+1}\Supset\cdots\Supset B_{L'}$ is of saddle-node type for some $L-b<L'\le L$;
\item for any $x\in B_{L'}$, we have
\[|R_{B_1}(x)-x|\ge \frac{|B_1|}{b};\] 
\item for each $1\le j\le L-1$, we have the Yoccoz equality:
\[\frac{1}{\eta k^2} <\frac{|B_{j}\setminus B_{j+1}|}{|B_1|}<\eta\frac{1}{k^2} \]
where $k=\min (j,L-j)$.

\end{enumerate}
\end{Lemma}

\subsection{Admissible intervals and transition maps}
We say an interval $T$ is an {\it admissible interval} if $T\in \mathcal M(\mathfrak E_0)$. An admissible interval $T'$ is called  a {\it pullback} of $T$ if $T'\in \mathcal M(T)$. More precisely, we say $T'$ is a $k$-pullback of $T$ if $F^k(T') \subset T$ and $F^k(\partial T')\subset \partial T $.\par
For an admissible interval $T$, let $\mathcal A_T:T\to \mathrm{int}(\I)$ be an orientation-preserving affine homeomorphism.
Let $T'$ be a $k$-pullback of $T$, then the {\it transition map} of $T$ and $T'$is defined as
\[G_{T,T'}:=\mathcal A_T\circ F^k\circ \mathcal A^{-1}_{T'}:\mathrm{int}(\I) \to \mathrm{int}(\I).\]
 Let  $\{G_s\}^k_{s=0}$ be the chain from $T'$ to $T$, that is,  a sequence of  open intervals such that $G_k=T$, $G_0=T'$ and $G_{s+1}$ is a component of $F^{-1}(G_s)$  for every $0\le s\le k-1$. For a critical point ${\bf c} \in \mathrm{Crit}(F)$, if ${\bf c}\notin G_s$ for every $1\le s\le k-1$,  then we say the transition map $G_{T,T'}$ is {\em short (with respect to $\bf c$)}, otherwise $G_{T,T'}$ is called  {\em long (with respect to {\bf c})}. For a long transition map, let $1\le m_1<\cdots<m_\ell\le k-1$ be all the integers such that $G_{m_j}\ni {\bf c}$, set $T_{\ell+1-j}=G_{m_j}$, then we have a canonical decomposition (with respect to $\bf c$): \[G_{T,T'}=G_{T,T_1}\circ G_{T_1,T_2}\circ G_{T_{\ell-1},T_\ell}\circ G_{T_\ell,T'}.\]\par

We consider the principal nest:
\[I_0:=(\alpha,-\alpha)\times\{0\}\ni {\bf c_0}, I_1=\L_{\bf c_0}(I_0), \cdots, I_{n}=\L_{\bf c_0}(I_{n}),\cdots,\]
and $I_{\infty}=\bigcap I_{n}$ is a periodic interval since $f$ is renormalizable with respect to $0$.\par
To describe the geometric properties of $F$, we need the following definition.
\begin{definition}For each admissible interval $I$, let $U(I)$ be the union of all the components of $D_I$ which intersect $I\cap \omega(\bf c_0)$. We say $I$ has $C$-bounded geometry if
\begin{enumerate}
\item $((1+2C^{-1})I-(1-2C^{-1})I) \cap \omega(\bf c_0)=\emptyset$;
\item for each component $J$ of $I\setminus \partial U(I)$, $|J|>C^{-1}|I|$.
\end{enumerate}
An admissible interval which satisfies condition~$(1)$ is called $C$-nice.
\end{definition}
\begin{Rm}As $f$ is infinitely renormalizable, for any ${{\bf c}}\in\mathrm{Crit}(F)$, $\omega({\bf c})=\omega({\bf c_0})= P(F)$, where $P(F)$ is the postcritical set of $F$.
\end{Rm}

%
Let $m(0)=0$ and let $m(1)<m(2)<\cdots<m(\kappa)$ be all the non-central return moments, i.e., $R_{I_{m(k)-1}}$ displays a non-central return. The integer $\kappa$ is called the {\em height} of $F$.
\begin{Lemma}\label{nice}For any $q>0$ and $\rho>0$, there exists $C_1=C_1(\rho,\|F\|_{C^3})>0$  and $C_2=C_2(\rho,q,\|F\|_{C^3})>0$  with the following properties. If $\Lambda_{I_0}<\rho$, then
\begin{enumerate}
\item $1+C_1^{-1}<\lambda_{I_{m(k)}}<\rho$ for all $0\le k\le \kappa$;
\item $I_0$ is $C_1$-nice;
\item  for any $t\in \mathbb N$ with $\inf\limits_{k}|t-m(k)|\le q$, $I_t$ has $C_2$-bounded geometry.
\end{enumerate}
\end{Lemma}
\begin{proof}
Theorem~\ref{realbounds} implies $(1)$. Statements $(2)$ and $(3)$ follows from \cite[Proposition~5.10]{C2Density} and \cite[Theorem~5.5]{C2Density} respectively.
\end{proof}
Given an  interval $J\subset \mathbb R$, let $\C_{J}:=\C\setminus(\R\setminus J)$ denote the plane slit along two rays. 
Following Shen\cite{C2Density}, we define the Epstein class as following. For any $C>0$ and $\eta>0$, the class $\mathcal K(C,1+\eta)$ consists of diffeomorphisms $\phi:\I \to \I$ with following properties:
\begin{enumerate} 
\item the $C^{1+1/2}$-norm of $\phi$ is at most $C$;
\item $\phi^{-1}|_{\mathrm{int}(\I)}$ extends to a real symmetric $(1+\eta)$-qc map from $\C_{\mathrm{int}(I)}$ into itself.
\end{enumerate} 
For any $u\in [-1/2,1/2]$, let $Q_u(z)=u(z^2-1)+z$. For any $v\in (0,2]$, let $P_v(z)=v(z^2-1)+1$. Let $\mathcal{SE}(C,1+\eta, M)$ denote the set of all functions $\Phi:\I\to \I$ which can be written as
\[\Phi=\psi_m\circ \phi_m\circ\cdots\psi_1\circ\phi_1\]
for some $m\le M$, where for each $1\le j\le m$, $\phi_j \in \mathcal K(C,1+\eta)$; and $\psi_j=Q_{u_j}$ for some $u_j\in [-1/2,1/2]$ or $\psi_j=P_{v_j}$ for some $v_j\in [1/C,2]$.  We say $\Phi:\I\to \I$ is in {\em the Epstein class} if $\Phi \in \mathcal{SE}(C,1,M)$ for some $C>0$ and $M>0$.
\begin{Rm}
For any $C>0$,$\eta>0$ and $N>0$, $\mathcal{SE}(C,1+\eta,M)$ is compact in $C^{1}$-topology. If $\Phi_k \in \mathcal{SE}(C,1+1/k,M)$ converges to $\Phi$ in $C^1$-topology, then $\Phi \in \mathcal{SE}(C,1,M)$.
\end{Rm}
\begin{Lemma}\label{lem:conformal} If $F^k:J \to F^k(J)$ is a diffeomorphism, then $F^{-k}:F^k(J) \to J$ extends to a conformal map from $\C_{F^k(J)}$ into $\C_{J}$.
\end{Lemma}
\begin{proof}Since $F$ is infinitely anti-renormalizable, there exists a sequence $\{H_j\}_{j=1}^{\infty}$ of infinitely renormalizable extended maps with following properties:
\begin{itemize}
\item there exist positive integers $m_1,m_2,\cdots$ and multi-intervals $J_1,J_2,\cdots$ such that $H_j^{m_j}|_{J_j}$ is affinely conjugate to $F$;
\item each fiber of $J_i$ has length less than $ \lambda^{-j}$ with $\lambda>1$ for all $j=1,2,\cdots$, where $\lambda$ is given by Theorem~\ref{realbounds}.  
 \end{itemize}
  Let $\Psi_j$ be the affine conjugacy between $F$ and $H_j^{m_j}|_{J_j}$. Then $H^{km_j}_j:\Psi_j(J) \to \Psi_j(F^k(J))$ is a diffeomorphism, by \cite[Proposition~5.7]{C2Density}, $H^{-km_j}_j:\Psi_j(F^k(J))\to \Psi_j(J)$ can extend to a $\exp(O(\lambda^{-j}))$-qc map from $\C_{\Psi_j(F^k(J))}$ into $\C_{\Psi_j(J)}$. Use the affine conjugacy, we conclude $F^{-k}:F^k(J) \to J$ can extend to a $\exp(O(\lambda^{-j}))$-qc map $\Phi_j$ from $\C_{F^k(J)}$ into $\C_{J}$ for all $j\in \mathbb N$. By the compactness of normalized $K$-qc maps, $\Phi_j$ converges uniformly to a conformal map $\Phi$ from $\C_{F^k(J)}$ into $\C_{J}$. Clearly, $\Phi|_{F^k(J)}=F^{-k}$.
\end{proof}

\begin{Lemma}\label{lem:epstein}There exists $C=C(\delta,M)>0$ with the following property. Let $\{G'_s\}_{s=0}^k$ and $\{G_s\}_{s=0}^k$ be chains satisfying:
\begin{itemize}
 \item $G'_s\supset G_s$ for all $0\le s \le k$ and $G_0 \cap \omega({\bf c_0})\ne \emptyset$;
 \item the multiplicity of $\{G'_s\}_{s=0}^k$ is at most $M$;
 \item  $G'_k\supset (1+2\delta)G_k$ and $|f^{k}(G_0)|\ge \delta |G_k|$.
\end{itemize}
For any $0\le s\le k$, let $\gamma_s:\overline{G_s} \to \I$ be the orientation-preserving affine homeomorphism. Then the map
\[\gamma_k\circ F^k\circ \gamma^{-1}_0:\I\to \I\]
belongs to the Epstein class $\mathcal{SE}(C,1,2M{\bf N})$.
\end{Lemma}
\begin{proof} We will use a similar argument in the proof of Lemma~\ref{lem:conformal}. Since $F$ is infinitely anti-renormalizable, there exists a sequence $\{H_j\}_{j=1}^{\infty}$ of infinitely renormalizable extended maps with following properties:
\begin{itemize}
\item there exist positive integers $m_1,m_2,\cdots$ and multi-intervals $J_1,J_2,\cdots$ such that $H_j^{m_j}|_{J_j}$ is affinely conjugate to $F$;
\item each fiber of $J_i$ has length less than $ \lambda^{-j}$ with $\lambda>1$ for all $j=1,2,\cdots$, where $\lambda$ is given by Theorem~\ref{realbounds}.  
 \end{itemize}
  Let $\Psi_j$ be the affine conjugacy between $F$ and $H_j^{m_j}|_{J_j}$. Put $\widehat{G}'_{j,s}=\Psi_j(G'_s)$ and 
$\widehat{G}_{j,s}=\Psi_j(G_s)$ for all $0\le s\le k$. Then for any $j\in \mathbb N$, two chains $\{\widehat{G}'_{j,s}\}_{s=0}^k$ and $\{\widehat{G}'_{j,s}\}_{s=0}^k$ satisfy the following conditions:
\begin{itemize}
\item $\widehat{G}'_{j,s}\supset \widehat{G}_{j,s}$ for all $0\le s \le k$ and $\widehat{G}_{j,0} \cap \Psi_j(\omega({\bf c_0}))\ne \emptyset$;
 \item the multiplicity of $\{\widehat{G}'_{j,s}\}_{s=0}^k$ is at most $M$;
 \item  $\widehat{G}'_{j,k}\supset (1+2\delta)\widehat{G}_{j,k}$ and $|f^{k}(\widehat{G}_{j,0})|\ge \delta |\widehat{G}_{j,k}|$.
 \item $|\widehat{G}'_{j,k}|<\lambda^{-j}$.
\end{itemize}
By \cite[Proposition~5.8]{C2Density},  there is a constant $C>0$ with the following property. For any $\eta>0$, there exists $j_0\in \mathbb N$ such that 
\[\widehat{\gamma}_{j_0,k}\circ (H_j^{m_j}|_{J_j})^k\circ \widehat{\gamma}^{-1}_{j_0,0}:\I\to \I\]
belongs to $\mathcal{SE}(C,1+\eta,2M{\bf N})$, where $\widehat{\gamma}_{j_0,k}=\gamma_k\circ \Psi^{-1}_j$ and $\widehat{\gamma}_{j_0,0}=\gamma_0\circ \Psi^{-1}_j$. Thus $\gamma_k\circ F^k\circ \gamma^{-1}_0:\I\to \I$ belongs to $\mathcal{SE}(C,1+\eta,2M{\bf N})$. As $\eta$ is arbitrary, $\gamma_k\circ F^k\circ \gamma^{-1}_0:\I\to \I$ belongs to  the Epstein class $\mathcal{SE}(C,1,2M{\bf N})$.
\end{proof}

\begin{Lemma}\label{FR} There exists a constant $C'=C'(C)>0$ with the following property. Let ${\bf c}\in\mathrm{Crit}(F)$ and ${\bf c} \in T$ be a $C$-nice admissible interval.  If $T'=\mathcal L_{\bf c}(T)$ and $|R_{T}(T')|\ge{C}^{-1}|T|$, then  $T'$ is $C'$-nice and  the transition map $G_{T,T'} \in \mathcal{SE}(C',1,8{\bf N})$.
\end{Lemma}
\begin{proof}Let $\{G_s\}_{s=0}^k$ be the chain from $T'$ to $T$ and $\{G'_s\}_{s=0}^k$ be the chain such that $G'_k=(1+2C^{-1})T$ and $G'_0 \supset T'$.\par
By \cite[Lemma~3.8]{C2Density}, the chain $\{G'_s\}_{s=0}^k$ has order at most $\bf N$ and multiplicity at most $4$. As $|F^k(G_0)|=|R_T(T')|>C^{-1}|T|=C^{-1}|G_k|$. It follows Lemma~\ref{lem:epstein} that there exitsts $C'>0$ such that $G_{T,T'} \in \mathcal{SE}(C',1,8{\bf N})$.\par
Now we prove $T'$ is $C''$-nice for some $C''>0$. To this end, let $\{\tilde G_s\}_{s=0}^k$ be the chain such that $\tilde G_k=(1+C^{-1})T$ and $\tilde G_0 \supset T'$. A straightforward way to prove $T'$ has $C''$-nice property is using Lemma~\ref{distortion}, which is left to the readers. We will use another way, by using the compactness of the Epstein class, which also works for subsequent Lemmas. By a similar argument in the previous paragraph, we can conclude $\gamma_k\circ F^k\circ \gamma^{-1}_0: \I\to \I$ is in the Epstein class $\mathcal{SE}(C',1,8{\bf N})$, where $\gamma_k:\overline{\tilde G_k} \to \I$ and $\gamma_0:\overline{\tilde G_0}\to \I$ are orientation-preserving affine homeomorphisms. We shall prove $T'$ is well inside $\tilde G_0$, that is, $\tilde G_0$ contains a definite neighborhood of $T'$. For otherwise, there exists a sequence $\{F_j\}_{j=0}^{\infty}$ of extended maps with following properties:
\begin{itemize}
\item for any $j\in \mathbb N$, $T(F_j)$ is a $C$-nice admissible interval of $F_j$;
\item for any $j\in \mathbb N$, there exists $x_j\in T(F_j)$ and $z_j \in \partial \tilde G_0(F_j)$ such that $|x_j-z_j||T'(F_j)|^{-1} \to 0$ as $j\to \infty$;
\item for any $j\in \mathbb N$, $|F^{k_j}_j(x_j)-F^{k_j}_j(z_j)|>\frac{1}{2}C^{-1}|T(F_j)|$.
\end{itemize}

Let $\gamma_{j,0}:\overline{\tilde G_0(F_j)} \to \I$ and $\gamma_{j,k_j}:\overline{\tilde G_{k_j}(F_j)} \to \I$ be the orientation-preserving homeomorphisms and $\widehat{x}_j=\gamma_{j,0}(x_j)$, $\widehat{z}_j=\gamma_{j,0}(z_j)$. Without loss of generality, we can assume $\widehat{x}_j \to \widehat x \in \I$, $\widehat{z}_j \to \widehat z \in \I$ and $\Phi_j:=\gamma_{j,k_j}\circ F^{k_j}_j\circ \gamma^{-1}_{j,0}$ converges to some $\widehat{\Phi}$ in $C^1$-topology. Then by the properties of $\{F_j\}$ we have
\begin{itemize}
\item  $|\widehat{x}_j-\widehat{z}_j| \to 0$ as $j\to \infty$, and then $\widehat x=\widehat z$;
\item  $|\Phi_j(\widehat{x}_j)-\Phi_j(\widehat{z}_j)|>\displaystyle{\frac{1}{2}C^{-1}\frac{C}{1+C}=\frac{1}{2(1+C)}}$.
\end{itemize}
Then by the uniform convergence, it follows $0=|\widehat{\Phi}(\widehat x)-\widehat{\Phi}(\widehat z)|>\frac{1}{2(1+C)}$, which is ridiculous. 
Thus $T'$ is well inside $\tilde G_0$, which implies there exists $C''>0$ such that $(1+2/C'')T'\setminus T' \cap \omega({\bf c_0})=\emptyset$ since $\tilde G_0\setminus T' \cap \omega({\bf c_0})=\emptyset$.\par
 A similar argument shows that $T'\setminus (1-2/C'')T'  \cap \omega({\bf c_0})=\emptyset$. Hence, $T'$  is $C''$-nice. Enlarge $C'$ so that $C'>C''$, and we are done.
 \end{proof}
 
\begin{Coro}\label{coro:FR} There exists a constant $C'=C'(C,d)>0$ with the following property. Let ${\bf c}\in\mathrm{Crit}(F)$ and ${\bf c} \in T$ be a $C$-nice admissible interval.  If $T'=\mathcal L^d_{\bf c}(T)$ and  $R_T\circ\cdots \circ R_{\mathcal L^{d-1}_{\bf c}(T)}(T')$ has length at least $C^{-1}|T|$,  then $T'$ is $C'$-nice and the transition map $G_{T,T'} \in \mathcal{SE}(C',1,8d{\bf N})$.
\end{Coro}
\begin{proof}It follows immediately by Lemma~\ref{FR} and induction.
\end{proof}

\subsection{Compactness for transition maps}

\begin{Lemma}\label{4.1}For any $\theta>0$,  there exists a constant $\xi=\xi(\theta,\|F\|_{C^3})>0$ such that the following holds. Let $B, B' \in \mathcal M(I_0)$, assume $B'= \mathcal L_x(B)$ for some $x \in B$ and $B \supset (1+2\xi)B'$. Then
\[\mathcal L_{\bf c_0}(B) \supset (1+2\theta)\mathcal L_{\bf c_0}(B').\]
\end{Lemma}
\begin{proof}See~\cite[Proposition~4.1]{C2Density}.
\end{proof}

\begin{Lemma}\label{saddle-node}Let $T_1,T_2,\cdots,T_L$ be a maximal central cascade with respect to some ${\bf c} \in \mathrm{Crit}(F)$.
 Assume there exists a positive integer $C>1$ such that
 \begin{itemize}
 \item $T_1$ is $C$-nice with limit scaling factor $\Lambda_{T_1}<C$;
 \item $T_1\supset (1+2C^{-1})T_2$ and $\min (\ell,L-\ell)\le C$;
  \item $|T_1|\le C|R^{\ell}_T(T_\ell)|$.
  \end{itemize}
 Then there exists a positive integer $C'=C'(C,\|F\|_{C^3})$  such that the transition map $G_{T_1,T_\ell}$ belongs to the Epstein class $\mathcal{SE}(C',1,8C'{\bf N})$.
\end{Lemma}

\begin{proof}

Let $b=b(C^{-1},C,\|F\|_{C^3})$ and $L-b<L'\le L$ be as in Lemma~\ref{sn}. Since $T_1$ is $C$-nice and  $|T_1|\le C|R^{\ell}_T(T_\ell)|$, by Corollary~\ref{coro:FR},  if $\ell\le \max (3b,C)$ is not large,  then there exists a positive integer $C'>\max(3b,C)$ such that $G_{T_1,T_\ell} \in \mathcal{SE}(C',1,8\ell {\bf N}) \subset \mathcal{SE}(C',1,8C'{\bf N})$.\par
Now we suppose that $L\ge \ell>\max(3b,C)$. By Lemma~\ref{sn},  $T_b, T_{b+1},\cdots, T_{L'}$ is of  saddle-node type and  for any $x\in T_{L'}$, 
\[|R_{T_1}(x)-x|\ge \frac{1}{b}|T_1|.\]
Clearly, $x, R_{T_1}(x),\cdots, R^{b}_{T_1}(x)$ lie in order.  Thus, we have $R^{b}_{T_1}(T_\ell) \subset T_{\ell-b}\setminus T_{L'}$.  
Let $J$ be the component of $ T_{\ell-b}\setminus T_{L'}$ containing $R^{b}_{T_1}(T_\ell)$. Then $R^{\ell-2b}_{T_1}|_J$ maps $J$ diffeomorphicly onto a component $\widehat J$ of $T_b\setminus T_{L'-\ell+2b}$ since all the critical points of $R_{T_b}$ are contained in $T_{L'}$.  By Lemma~\ref{lem:conformal}, $(R^{\ell-2b}_{T_1}|_{J})^{-1}: \widehat J \to J$ can extend to a conformal mapping from $\C_{\widehat J}$ onto $\C_{J}$.
We shall prove that the diffeomorphism $R^{\ell-2b}_{T_1}|_{J}$ has uniformly bounded distortion, which implies $\gamma_1 \circ R^{\ell-2b}_{T_1}|_{J}\circ \gamma^{-1}_0 \in \mathcal{K}(C'',1)$ for some $C''>0$ where $\gamma_0:\overline{J}\to \I$ and $\gamma_1:\overline{\widehat J}\to \I$ are orientation preserving affine homeomorphisms.

Since $L'-\ell+2b\le L-\ell+2b\le C+2b$, by Lemma~\ref{sn} or Yoccoz's Lemma, there exists $C_1=C_1(C,b)>1$ such that
\[\frac{1}{C_1}|T_1|\le |T_{L'-\ell+2b-1}-T_{L'-\ell+2b}|.\]

By Corollary~\ref{coro:FR}, there exists $C'>1$ such that $G_{T_1,T_b}$  lies in a compact set $\mathcal{SE}(C',1,8b{\bf N})$ and $T_b$ is $C'$-nice. Thus we can extend $R^{\ell-2b}_{T_1}|_J$ to be a diffeomorphism onto a $C_2$-neighborhood of $\widehat J$
, where $C_2=\min(1/C',1/C_1)$.  By real Koebe principle, $R^{\ell-2b}_{T_1}|_J$ has uniformly bounded distortion.\par
Finally, since $G_{T_1,T_{b}}$ and $G_{T_{\ell-b},T_{\ell}}$ both lie in $\mathcal{SE}(C',1,8b{\bf N})$, we conclude

$$G_{T_1,T_\ell}=G_{T_1,T_N}\circ \gamma_1\circ R^{\ell-N'-N}_{T_1}\circ\gamma^{-1}_0\circ G_{T_{\ell-N'},T_{\ell}}$$ belongs to $\in \mathcal{SE}(C',1,8C'{\bf N})$ by enlarging $C'$.
\end{proof}

%
%
The following lemmas will play an important role in the proof of the Key Lemma.

\begin{Lemma}\label{saddlebound} Let ${\bf c} \in \mathrm{Crit}(F)$. Assume ${\bf c}\in T$ is an admissible interval with \[{|T|}<\Lambda{|\mathcal L_{\bf c}(I_{\infty})|}\] for some $\Lambda>0$ and ${\bf c_0}\in T'$ is a $k$-pullback of $T$ with $k\le {\bf N}p$.  Let $\{G_s\}_{s=0}^{k}$ be the chain from $T'$ to $T$.  If $G_{s_\ell}=T_1,T_2,\cdots,T_\ell=G_{s_1}$ is a central cascade in the decomposition of the chain $\{G_s\}_{s=0}^{k}$ with following properties:
\begin{enumerate}
\item if $T_1, T_2,\cdots, T_L$ is a maximal central cascade, then either $L=\ell$ or $T_{\ell+1} \neq G_s$ for all $0\le s<s_1$; 
\item there exists positive integer $C>1$ such that $T_1$ is $C$-nice with scaling factor $\lambda_{T_1}>1+2C^{-1}$ and \[{|T_1|}<C{|F^{s_\ell}(I_{\infty})|}, \]
\item either  ${\bf c_0}\notin G_s$ for $s_1\le s \le s_\ell$ or $T_1,T_2,\cdots,T_\ell$ is a cascade with respect to ${\bf c_0}$,
\end{enumerate}
then there exists a positive integer $C'=C'(\Lambda,C,\|F\|_{C^3})$  such that 
\begin{itemize}
\item $T_\ell$ is $C'$-nice with $\lambda_{T_\ell} >1+2C'^{-1}$;
\item $|T_\ell|<C'|F^{s_1}(I_{\infty})|$;
\item  $G_{T_1,T_\ell} \in \mathcal{SE}(C',1,8C'{\bf N})$. 
\end{itemize}
\end{Lemma}

\begin{proof}
First, we prove there exists $\rho=\rho(\Lambda,\|F\|_{C^3})$ such that the limit scaling factor $\Lambda_{T_1}$ of $T_1$ is less than $\rho$.
 Fix some constant $\theta>\Lambda$ and let $\xi=\xi(\theta,\|F\|_{C^3})$ be given by Lemma~\ref{4.1}.  We claim that $\Lambda_{T_1} \le \rho:=2\xi+1$. For otherwise, there exists a critical point ${\bf c'}$ and an  admissible interval $I\ni {\bf c'}$ such that $I\supset (1+2\xi)\mathcal L_{\bf c'}(I)$. By Lemma~\ref{4.1}, 
\[\Lambda \mathcal L_{\bf c}(I_{\infty})\subset (1+2\theta)\mathcal L_{\bf c}(I_{\infty}) \subset  (1+2\theta)\mathcal L_{\bf c}(\mathcal L_{\bf c'}(I)) \subset\mathcal L_{\bf c}(I) \subset T,\]
which contradicts with the assumption.\par 
Let us now consider the maximal central cascade $T_1,T_2,\cdots,T_L$. We shall prove  there exists $q=q(\Lambda,C,\|F\|_{C^3})$ such that $Q:=\min (\ell,L-\ell) \le q$.  \par
Set $J_i:=F^i(I_{\infty})$ for all $i\ge 1$.  Let $b=b(C^{-1}, \rho,\|F\|_{C^3})$, $L-b<L'\le L$ and $\eta=\eta(C^{-1},\rho,\|F\|_{C^3})$ be given by Lemma~\ref{sn}. Without loss of generality, we may assume $L> \ell>10b$. \par

{\bf Case 1.} $J_{s_1}\subset T_{L'}$. It follows from Lemma~\ref{sn} that $T_b,T_{b+1},\cdots, T_{L'}$ is a central cascade of saddle-node type and for any $x\in T_{L'}$, $$|R_{T_1}(x)-x|>\frac{1}{b}|T_1|.$$ 
Note that $R^b_{T_1}(x) \notin T_{L'}$ for any $x\in J_{s_1}$. Indeed, for such $x\in J_{s_1}\subset T_{L'}$, $ x, R_{T_1}(x), \cdots R^{b}_{T_1}(x)$ lie in order.  So if $R^{b}_{T_1}(x) \in T_{L'}$, then $|R^{b}_{T_1}(x)-x|\ge N'\times \frac{1}{b}|T_1|=|T_1|$, which is a contradiction.
Thus we obtain 
\begin{align*}
J_{s_\ell}=R_{T_1}^{\ell-1}(J_{s_1})&=R_{T_1}^{\ell-1-b}\circ R_{T_1}^{b}(J_{s_1})\\ 
&\subset R_{T_1}^{\ell-1-b}(T_{L'-b}\setminus T_{L'})=T_{L'-\ell+1}\setminus T_{L'-\ell+1+b}.
\end{align*}
Hence, 
\[C^{-1}<\frac{|J_{s_\ell}|}{|T_1|}\le\frac{|T_{L'-\ell+1}\setminus T_{L'-\ell+1+b}|}{|T_1|}. \]
By Lemma~\ref{sn}, we have

$$\frac{|T_{L'-\ell+1}\setminus T_{L'-\ell+1+b}|}{|T_1|} \asymp \frac{b\eta}{Q^2}.$$
So $Q$ is bounded in terms of $C$ and $\|F\|_{C^3}$. 

{\bf Case 2.} Set $M=\max\{2b, [\sqrt{C\eta}]+1\}$. If $J_{s_1} \subset T_{\ell+m}\setminus T_{\ell+m+1}$ for some $m>M$ with $\ell+m<L'$, then $J_{s_\ell}=R_{T_1}^{\ell-1}(J_{s_1}) \subset T_{m+1}\setminus T_{m+2}$. It follows from Lemma~\ref{sn}, that
\[C^{-1}<\frac{|J_{s_\ell}|}{|T_1|}\le \frac{|T_{m+1}\setminus T_{m+2}|}{|T_1|}\le \max \{\frac{\eta}{(m+1)^2}, \frac{\eta}{(L-m-1)^2}\}.\]
As $\displaystyle{\frac{\eta}{(m+1)^2} \le \frac{\eta}{M^2}<C^{-1}}$, we have $(L-m-1)^2<C\eta$, and then
\[\ell\le L-m<\sqrt{C\eta}+1.\]

{\bf Case 3.} Now suppose $J_{s_1} \subset T_{\ell+m}\setminus T_{\ell+m+1}$ for some $m\le M$.  Take a maximal integer $s_1\le r<s_1+{\bf N}p$ such that $J_r\subset T_1$, where $p$ is the renormalization period of $f$. Such an $r$ exists since $J_{s_\ell} \subset T_1$. Clearly, $R_{T_1}(J_r)=J_{s_1+{\bf N}p}
=J_{s_1}$.

\medskip
{\bf Claim.} $J_r \subset T_1\setminus T_2$.

\medskip
Let $u=s_\ell-s_{\ell-1}$, then $R_{T_1}|_{T_2}=F^u|_{T_2}$. If the claim fails, i.e., $J_r \subset T_2=G_{s_{\ell-1}}$, then $T_{\ell+1}$ is the component of $F^{-u}(T_\ell)$ containing $J_r$ and $u+r={\bf N}p+s_1$. If $u\le s_1$, then $G_{s_1-u}$ is the component of $F^{-u}(T_{\ell})$ containing $J_r$. This implies $T_{\ell}=G_{s_1-u}$, which contradicts with condition (1). Thus, $u>s_1$. However, it is also impossible. Indeed, if $u>s_1$, then $r={\bf N}p+s_1-u<{\bf N}p$. Then $G_{s_{\ell-1}+{\bf N}p-r} \supset J_{{\bf N}p} \ni {\bf c_0}$ since $J_r \subset G_{s_{\ell}-1}$. A direct computation shows that
\[s_{\ell-1}+{\bf N}p-r=s_{\ell-1}+s_1-u=s_{\ell-1}-s_1-(s_{\ell}-s_{\ell-1})=s_1-s_\ell \in [0,s_1] \cap \mathbb N.\]
By condition (3), $T_1,T_2,\cdots,T_{\ell}$ should be a central cascade with respect to ${\bf c_0}$. Then by the definition of $r$, $r\ge {\bf N}p$ since $J_{{\bf N}p} \subset T_1$. This is a contradiction, as  $u+r={\bf N}p+s_1$ and $u>s_1$. Hence, the claim follows.

\medskip

Let $D\supset J_r$ be the component of the  first return domain to $T_1$ and let $F^v|_D:D\to T_1$ be the corresponding return map. By the definition of $r$, we have $v={\bf N}p+s_1-r$.

\par

Set $D':=(F^v|_D)^{-1}(T_{\ell+m}\setminus T_{\ell+m+1})$, by the Markov property, $F^v(\mathcal L_{J_r}(D))$ must be contained in $T_{\ell+m}\setminus T_{\ell+m+1}$.
 It follows $D'\supset \mathcal L_{J_r}(D) $. We claim $(1+2\xi)\mathcal L_{J_r}(D) \not\subset D$. For otherwise, by Lemma~\ref{4.1}, $\mathcal L_{\bf c_0}(D) \supset (1+2\theta) \mathcal L_{\bf c_0}(\mathcal L_{J_r}(D))$. So
\[\frac{1}{\Lambda}\le\frac{|I_{\infty}|}{|T|}\le \frac{| \mathcal L_{\bf c_0}(\mathcal L_{J_r}(D))|}{|\mathcal L_{\bf c_0}(D)|}<\frac{1}{1+2\theta}<\frac{1}{2\Lambda}.\]
This is absurd. By Lemma~\ref{sn}, we have
\[\frac{|T_{\ell+m}\setminus T_{\ell+m+1}|}{|T_1|} \asymp \frac{1}{(\min \{\ell,L-\ell\})^2}=\frac{1}{Q^2}.\]
It follows from Lemma~\ref{distortion} that there exists $\delta_1=\delta(Q)$ such that $(1+2\delta_1)\mathcal L_{J_r}(D) \subset D$ and $\delta_1(Q) \to \infty$ as $Q\to \infty$. As $\delta_1 \le \xi$, there exists $q=q(\xi,C,\|F\|_{C^3})=q(\Lambda,C,\|F\|_{C^3})$ such that $\min (\ell,L-\ell)=Q\le q$. Thus we are done.\par
Hence, $T_1$ satisfies the following properties:
\begin{enumerate}
\item $T_1$ is $C$-nice with limit scaling factor $\Lambda_1\le \rho$;
\item $T_1\supset (1+2C^{-1})T_2$;
\item $C|R^{\ell}_{T_1}(T_\ell)|\ge C|F^{s_\ell}(I_{\infty})|>|T_1|$;
\item $Q=\min(\ell,L-\ell)\le q$.
\end{enumerate}
By Lemma~\ref{saddle-node}, there exists a positive integer $C'>1$ such that $G_{T_1,T_\ell} \in \mathcal{SE}(C',1,8C'{\bf N})$.  As ${|T_1|}<C{|F^{s_\ell}(I_{\infty})|}$ and $T_1$ is $C$-nice, it follows easily there exists $C''=C''(\Lambda,C,\|F\|_{C^3})$ such that ${|T_\ell|}<C''{|F^{s_1}(I_{\infty})|}$  and $T_\ell$ is $C''$-nice from the fact that $G_{T_1,T_\ell}$ lies in a compact set. The argument is similar to that we used in the proof of Lemma~\ref{FR}. We only prove ${|T_\ell|}<C''{|F^{s_1}(I_{\infty})|}$. Indeed, if there does not exist such $C''$, then there exists a sequence $(F_i)$ of extended maps with following properties:
\begin{itemize}
\item $\{ G_{i}:=G_{T_1(F_i),T_{\ell_i}(F_i)}\}$ converges uniformly to some map $G$;
\item there exist $x_1(i),x_2(i) \in F^{s_{\ell_i}}_i(I_{\infty}(F_i))$  such that $\widehat x_1(i)-\widehat x_2(i) \to 0$ as $i\to \infty$, where $\widehat x_1(i)=\mathcal A_{T_{\ell_i}(F_i)}(x_1(i))$ and $\widehat x_2(i)=\mathcal A_{T_{\ell_i}(F_i)}(x_2(i))$;
\item $|F^{s_{\ell_i}-s_{1_i}}_i(x_1(i))-F^{s_{\ell_i}-s_{1_i}}_i(x_2(i))|=|F^{s_{1_i}}_i(I_{\infty}(F_i))|>C^{-1}|T_1(F_i)|$, in other words, $|G_i(\widehat x_1(i))-G_i(\widehat x_2(i))|>C^{-1}$.
\end{itemize}
Without loss of generality, we can assume $\widehat x_1(i)\to \widehat x \in [-1,1]$. Since $|G_i(\widehat x_1(i))-G_i(\widehat x_2(i))|>C^{-1}$, take a limit, we obtain $0=|G(\widehat x)-G(\widehat x)|>C^{-1}$. This is absurd.\par
As $Q=\min(\ell,L-\ell)$ is uniformly bounded, by (3) of Lemma~\ref{sn}, $\lambda_{T_\ell}$ is uniformly  bounded both below and above. Hence, enlarge $C'$ to be large enough, then the conclusions of this lemma follow.

\end{proof}

\begin{Lemma}[Long transition maps]\label{long}
Assume ${\bf c}\in T$ is an admissible interval with \[{|T|}<\Lambda{|\mathcal L_{\bf c}(I_{\infty})|}\]
for some $\Lambda>0$ and $T'\ni{\bf c_0}$ is a $k$-pullback of $T$ with $k\le {\bf N}p$.  If there exists a positive integer $C>1$ such that
\begin{enumerate}
\item$T$ is $C$-nice with $\lambda_T>1+2C^{-1}$;
\item $|T|<C|F^k(I_{\infty})|$. 
\end{enumerate}
Then there exists a positive integer $C'=C'(\Lambda,C,\|F\|_{C^3})>1$ such that $G_{T,T'} \in \mathcal{SE}(C',1,8C'{\bf N})$.
\end{Lemma}

\begin{proof}
First, a similar argument in the proof of Lemma~\ref{saddlebound} shows that  there exists $\rho=\rho(\Lambda,\|F\|_{C^3})$ such that the limit scaling factor $\Lambda_{T}$ of $T$ is less than $\rho$.

Let $\{G_s\}_{s=0}^k$ be the chain from $G_0=T'$ to $G_k=T$. 

\medskip
{\bf Statement ($M$).}{\em Let $Y=G_{y},Y'= G_{y'}$ $(0\le y'<y\le k)$ be two symmetric admissible intervals in the chain such that 
\begin{enumerate}
\item $Y$ is $C$-nice with $\lambda_{Y}>1+2C^{-1}$;
\item $|Y|<C|F^{y}(I_{\infty})|$.
\end{enumerate}
If $\# (\mathrm{Crit}(F) \bigcap (\bigcup\limits_{s=y'+1}^{y-1} G_s)\le M$, then there exists a positive integer $C'=C'(\Lambda, C,\|F\|_{C^3})>1$ such that
\begin{itemize}
\item $Y'$ is $C'$-nice with $\lambda_{Y'}>1+2C'^{-1}$;
\item $|Y'|<C'|F^{y'}(I_{\infty})|$;
\item $G_{Y,Y'}\in \mathcal{SE}(C',1,8C'{\bf N})$. 
\end{itemize} }

\medskip

{\em Proof of Statement $(0)$.} By Lemma~\ref{lem:conformal}, the diffeomorphism $$F^{-(y-y'-1)}:F^{y-y'}(Y')\to F(Y')$$ can extend to a conformal map from $\C_{F^{y-y'}(Y')}$ into $\C_{ F(Y')}$. Since $Y$ is $C$-nice, $F^{y-y'-1}|_{F(Y')}$ can be extended to a diffeomorphism onto $(1+2C^{-1})Y$. By Koebe's distortion theorem, $F^{y-y'-1}|_{F(Y')}$ has uniformly bounded distortion, which implies that $G_{Y,Y'} \in \mathcal{SE}(C',1,2)$ for some $C'>0$. Then it follows easily $|Y'|<C'|F^{y'}(I_{\infty})|$ by the compactness of $\mathcal{SE}(C',1,2)$.\par
Now we prove there exists $C''>0$ such that $\lambda_{Y'}>1+2C''^{-1}$.  Assume the critical point contained in $Y'$ is ${\bf c'}$. Let  $B'=\mathcal L_{\bf c'}(Y')$  and $r_{B'}$ be the return time from $B'$ to $Y'$, then $r_{B'}>y-y'$. For otherwise, $G_{y'+r_{B'}}\supset Y' $ contains a critical point, a contradiction. Thus, $F^{y-y'}(B')$ lies in a first return component $B$ to $Y$. Since $\lambda_Y>1+2C^{-1}$ and $\Lambda_Y\le \Lambda_T<\rho$ , we have $$\rho^{-1}|Y|<|\mathcal L_{\bf \hat c}(Y)|<\frac{C}{C+2}|Y|,$$  where ${\bf \hat c}$ is the critical point in $Y$. Thus, there exists $\widehat C>0$ such that $\widehat C^{-1}|Y|<|B|<(1-\widehat C^{-1})|Y|$. As $F^{y-y'-1}|_{F(Y')}$ has uniformly bounded distortion, there exists $C''>0$ such that $(1+2C''^{-1})|B'|<|Y'|$.\par
Enlarge $C'$ so that $C'>C''$, the conclusion of Statement (0) follows.

\medskip
{\em Statement $(M-1)$ $\Rightarrow$ Statement $(M)$.}

{\bf Case 1.} ${\bf c_0}\in \bigcup\limits_{s=y'+1}^{y-1} G_s$. Let $y'<v<y$ and $y'<v'<y$ be the largest and smallest integer such that ${\bf c_0}\in G_{v'} \subset G_v$ respectively.  Consider the canonical decomposition of $G_{G_v,G_{v'}}$ with respect to ${\bf c_0}$, that is, \[G_v=T^1_1\Supset \cdots \Supset T^1_{L_1} \Supset T^2_1\cdots\Supset T^2_{L_2}\Supset\cdots\Supset T^\ell_1\Supset\cdots\Supset T^\ell_{L_\ell}=G_{v'},\]
 where $T^j_1,\cdots,T^j_{L_j}$ is a central cascade satisfying condition (1) in Lemma~\ref{saddlebound}, $j=1,\cdots,\ell$. 
As $|T|<\Lambda|\mathcal L_{\bf c}(I_{\infty})|$, by Theorem~\ref{realbounds} and Lemma~\ref{4.1}, we know $\ell$ is bounded in terms of $\Lambda$ and $\|F\|_{C^3}$.
Let  $s^{j}_i$ be the integer such that $T^{j}_i=G_{s^j_i}$ for all $1\le j\le \ell$, $1\le i\le L_j$. By Statement $(0)$,  there exists a positive integer $\widetilde C_1>1$ such that
\begin{itemize}
\item[P1(1).] $T^1_{1}$ is $\widetilde C_1$-nice with $\lambda_{T^1_1}>1+2\widetilde C_1^{-1}$;
\item[P2(1).] $|T^1_1|<\widetilde C_1|F^{s^1_{1}}(I_{\infty})|$;
\item[P3(1).]  $G_{Y,G_v} \in \mathcal{SE}(\widetilde C_1, 1, 8\widetilde C_1{\bf N})$.
\end{itemize}
 It follows from Lemma~\ref{saddlebound} that 

\begin{itemize}
\item[H1(1).] $T^1_{L_1}$ is $\widehat C_1$-nice with $\lambda_{T^1_{L_1}}>1+2\widehat C_1^{-1}$;
\item[H2(1).] $|T^1_{L_1}|<\widehat C_1|F^{s^1_{L_1}}(I_{\infty})|$;
\item[H3(1).]  $G_{T^1_1,T^1_{L_1}} \in \mathcal{SE}(\widehat C_1,1,8\widehat C_1{\bf N})$.
\end{itemize}
By Statement $(M-1)$, there exists a positive integer $\widetilde C_2>1$ such that
\begin{itemize}
\item[P1(2).] $T^2_{1}$ is $\widetilde C_2$-nice with $\lambda_{T^2_1}>1+2\widetilde C_2^{-1}$;
\item[P2(2).] $|T^2_1|<\widetilde C_2|F^{s^2_{1}}(I_{\infty})|$;
\item[P3(2).]  $G_{T^1_{L_1},T^2_1} \in \mathcal{SE}(\widetilde C_2,1,8\widetilde C_2{\bf N})$.
\end{itemize}
By induction, we can obtain two finite sequences of positive integers $\{\widetilde C_j\}_{j=2}^{\ell}$ and  $\{\widehat C_j\}_{j=2}^{\ell}$
such that for all $2\le j\le \ell$, the following holds.
\begin{itemize}
\item[P1($j$).] $T^{j}_{1}$ is $\widetilde C_{j}$-nice with $\lambda_{T^{j}_1}>1+2\widetilde C_{j}^{-1}$;
\item[P2($j$).] $|T^{j}_1|<\widetilde C_{j}|F^{s^{j}_{1}}(I_{\infty})|$;
\item[P3($j$).]  $G_{T^{j-1}_{L_1},T^{j}_1} \in \mathcal{SE}(\widetilde C_{j},1,8\widetilde C_{j}{\bf N})$;
\item[H1($j$).] $T^j_{L_j}$ is $\widehat C_j$-nice with $\lambda_{T^j_{L_j}}>1+2\widehat C_j^{-1}$;
\item[H2($j$).] $|T^j_{L_j}|<\widehat C_j|F^{s^j_{L_j}}(I_{\infty})|$;
\item[H3($j$).]  $G_{T^j_1,T^j_{L_j}} \in \mathcal{SE}(\widehat C_j,1,8\widehat C_j{\bf N})$.
\end{itemize}
Consider the transition map $G_{G_{v'},Y'}$, by Statement $(M-1)$, there exists a positive integer $\widetilde C_{\ell+1}$ such that
\begin{itemize}
\item $Y'$ is $\widetilde C_{\ell+1}$-nice with $\lambda_{Y'}>1+2\widetilde C_{\ell+1}^{-1}$;
\item $|Y|<\widetilde C_{\ell+1}|F^{y'}(I_{\infty})|$;
\item  $G_{G_{v'},Y'} \in \mathcal{SE}(\widetilde C_{\ell+1}, 1, 8\widetilde C_{\ell+1}{\bf N})$.
\end{itemize}

Let $C'=\widetilde C_{\ell+1}+\Sigma_{j=1}^{\ell}(\widehat C_j+\widetilde C_{j})$, then the conclusions of Statement $(M)$ follow.

{\bf Case 2.} ${\bf c_0}\notin \bigcup\limits_{s=y'+1}^{y-1} G_s$. Let $y'<v<y$ be the largest integer such that $\mathrm{Crit}(F)\cap G_v\ne \emptyset$. Assume the critical point lying in $G_v$ is ${\bf c}$ and let $y'\le v'\le v$ be the smallest integer such that $G_{v'} \ni {\bf c}$. Consider the canonical decomposition of $G_{G_v,G_{v'}}$ with respect to ${\bf c}$, that is, \[G_v=T^1_1\Supset \cdots \Supset T^1_{L_1} \Supset T^2_1\cdots\Supset T^2_{L_2}\Supset\cdots\Supset T^\ell_1\Supset\cdots\Supset T^\ell_{L_\ell}=G_{v'},\]
 where $T^j_1,\cdots,T^j_{L_j}$ is a central cascade satisfying condition (1) in Lemma~\ref{saddlebound}, $j=1,\cdots,\ell$.  Then the proof is essentially the same as that of Case 1.
\end{proof}

\begin{Lemma}\label{compactness} Let $T$ be a component of $\mathfrak E_0$ and $T'\ni {\bf c_0}$ be an $({\bf N}p-b)$-pullback of $T$ for some $0\le b\le {\bf N}p$. If there exists $\Lambda>0$ such that $|I_0|<\Lambda|I_{\infty}|$, then there exists a positive integer $C'=C'(\Lambda,b,\|F\|_{C_3})$ such that $G_{T,T'} \in \mathcal{SE}(C',1,8C'{\bf N})$.
\end{Lemma}
\begin{proof}
First, we show there exists a constant $\Lambda'>0$ such that $\Lambda_{\mathfrak E_0}\le \Lambda'$. Fix some constant $\theta>\Lambda$ and let $\xi=\xi(\theta,\|F\|_{C^3})$ be given by Lemma~\ref{4.1}.  Then  $\Lambda_{\mathfrak E_0} \le \Lambda':=2\xi+1$. Indeed,  if $\Lambda_{\mathfrak E_0} > \Lambda'$ then there exists a critical point ${\bf c}$ and an  admissible interval $I\ni {\bf c}$ such that $I\supset (1+2\xi)\mathcal L_{\bf c}(I)$. By Lemma~\ref{4.1}, 
\[\Lambda I_{\infty}\subset (1+2\theta)\mathcal L_{\bf c_0}(I_{\infty}) \subset  (1+2\theta)\mathcal L_{\bf c_0}(\mathcal L_{\bf c}(I)) \subset\mathcal L_{\bf c_0}(I).\]
By Markov property, $\mathcal L_{\bf c_0}(I) \subset I_0$, thus $|I_0|\ge \Lambda|I_\infty|$, 
which contradicts with the assumption.\par
Recall that $F$ is bi-infinitely renormalizable with renormalization period large than $2{\bf N}$. As $\Lambda_{\mathfrak E_0}\le \Lambda'$,  by \cite[Theorem~5.6]{C2Density}, $|T|>C^{-1}$ and $T$ has $C$-bounded geometry for some $C>0$. Let $\{G_s\}_{s=0}^{{\bf N}p-b}$ be the chain from $T'$ to $T$ and $s_0$ be the largest integer such that $G_{s_0} \cap \mathrm{Crit}(F)\ne \emptyset$. Put $T_0=G_{s_0}$. We shall prove there exists a positive integer $C_0>1$ such that 
\begin{enumerate}
\item $G_{T,T_0} \in \mathcal{SE}(C_0,1,8C_0{\bf N})$;
\item  $T_0$ is $C_0$-nice with $\lambda_{T_0}>1+2C^{-1}_0$
\item $|T_0|<C_0|\mathcal L_{\bf c}(I_{\infty})|$, where ${\bf c}$ is the critical point in $T_0$;
\item $|T_0|<C_0|F^{s_0}(I_{\infty})|$.
\end{enumerate} 
Then the conclusion will follow easily. Indeed, by Lemma~\ref{long}, (2),(3) and (4) implies there exists a positive integer $C''>1$ such that 
$G_{T_0,T'} \in \mathcal{SE}(C'',1,8C''{\bf N})$. Let $C'=C_0+C''$, then $G_{T,T'}\in \mathcal{SE}(C',1,8C'{\bf N})$.\par
Now we prove (1). To this end, we first show there exists $\widehat C>0$ such that $|F^{{\bf N}p-b}(I_{\infty})|>\widehat C^{-1}$. We claim $|I_0|>\widetilde C^{-1}$ for some $\widetilde C=\widetilde C(\|F\|_{C^3})>0$. For otherwise, there exists a sequence $\{F_j\}$ of extended maps such that
\begin{itemize}
\item $F^{\bf N}_j$ converges to a map $H$ in $C^3$-topology;
\item $(F^{\bf N}_j)'(\alpha(F_j),0)<-1$ and $(\alpha(F_j),0) \to (0,0)$ as $j \to \infty$.
\end{itemize}
This implies $-1\le H'(0,0)=0$, which is absurd. Let $x_0 \in \partial I_{\infty}$ be a periodic point of $F$, let $\widehat x_0, \widehat z_0 \in F^{{\bf N}p-b}(I_{\infty})$ such that $F^b(\widehat x_0)=x_0$ and $F^b(\widehat z_0)={\bf c_0}$. By the Mean Value Theorem, there exists $\widehat{\zeta} \in T$ such that $|\widehat x_0-\widehat z_0||(F^b)'(\widehat{\zeta})|=|x_0-{\bf c_0}|>\frac{1}{2}\widetilde C^{-1}$. As $|(F^b)'(\widehat{\zeta})|$ is uniformly bounded above in terms of $\|F\|_{C^3}$, $|\widehat x_0-\widehat z_0|$ has uniform lower bound. So there exists $\widehat C>0$ such that $|F^{{\bf N}p-b}(I_{\infty})|>\widehat C^{-1}$.\par
Put $u={\bf N}p-b-s_0$. By Lemma~\ref{lem:conformal}, the diffeomorphism $F^{-(u-1)}:T \to F(T_0)$ can extend to a conformal map form $\C_{T}$ into $\C_{F(T_0)}$. Since $T$ is $C$-nice, by real Koebe principle, $F^{u-1}$ has uniformly bounded distortion. This implies $\gamma_1\circ F^{u-1}\circ \gamma_0:\I \to \I$ belongs to $\mathcal K(C_1,1)$ for some $C_1>1$, where $\gamma_1:\overline{T}\to \I$ and $\gamma_0:\overline{F(T_0)} \to \I$ are orietation-preserving homeomorphism. Thus, $G_{T,T_0} \in \mathcal{SE}(C_1,1,2)$. Then (2) and (4) follows easily by a compactness argument which we used frequently in the proof of Lemma~\ref{FR} and ~\ref{saddlebound}. \par
To prove (3), we consider a nest $I'_0=T_0, I'_1=\mathcal L_{\bf c}(I'_0), \cdots, I'_{j+1}=\mathcal L_{\bf c}(I'_j),\cdots$. Let $m'(1)<m'(2)<m'(\kappa')$ be all the non-central return moments, i.e., $R_{I'_{m'(i)}}$ displays a non-central return. Note that $\kappa'$ is bounded in terms of $\Lambda$ and $\|F\|_{C^3}$. Indeed, for any $1\le j\le \kappa'$, let $V_j=\mathcal L_{\bf c_0}(I'_{m'(j)+1})$ and $U_j=\mathcal L_{\bf c_0}(I'_{m'(j)+2})$, then $V_{j+1}\subset U_j$. By Theorem~\ref{realbounds} and Lemma~\ref{4.1}, there exists $\lambda'>1$ such that $\lambda' U_j \subset V_j$. Hence, if $\kappa'$ is not bounded, then $|I_{\infty}|/|I_0|$ should be small, a contradiction. Since $\Lambda_{T} \le \Lambda_{\mathfrak E_0}<\Lambda'$, by Lemma~\ref{sn}, $|I'_{m'(\kappa')+1}| \asymp |I'_0|$. We claim $|\mathcal L_{\bf c}(I_{\infty})| \asymp |I'_{m'(\kappa')+1}|$. For otherwise, by Lemma~\ref{4.1}, $I_{\infty}=\mathcal L_{\bf c_0}(\mathcal L_{\bf c}(I_{\infty}))$ will be geometrically deep inside $\mathcal L_{\bf c_0}(I'_{m'(\kappa')+1}) \subset I_0$. This is a contradiction. So $|\mathcal L_{\bf c}(I_{\infty})| \asymp |I'_{m'(\kappa')+1}| \asymp |I'_0|=|T_0|$.
\end{proof}

Let us now recall the definition of Kozlovski-Shen-vanStrien's enhanced nest (see \cite[Section~8]{KSS} ).
\begin{Lemma}\label{KSS} Let $T\ni {\bf c}$ be an admissible interval. Then there is a positive integer $\nu$ with $f^{\nu}({\bf c})\in T$ such that the following holds. Let $T'$ be the component of $f^{-\nu}(T)$ containing ${\bf c}$ and let $\{G_j\}_{j=0}^{\nu}$ be the chain from $T'$ to $T$. Then
\begin{enumerate}
\item $\#\{0\le j\le \nu-1: G_j\cap \mathrm{Crit}(F) \ne \emptyset \}\le {\bf N}^2$;
\item $T'\cap \omega({\bf c_0})$ is contained in the component of $f^{-\nu}(\mathcal L_{f^{\nu}({\bf c})}(T))$.
\end{enumerate}
\end{Lemma}
\begin{proof}See \cite[Lemma~8.2]{KSS}.
\end{proof}
For an open set $B$ and a point $x\in B$, we use $\mathrm{Comp}_x(B)$ to denote the component of $B$ containing $x$. For each admissible interval $T\ni {\bf c}$, let $\nu=\nu(T)$ be the smallest integer 
with the properties specified by Lemma~\ref{KSS}. Following Kozlovski-Shen-vanStrien, we define
\[\mathcal G(T)=\mathrm{Comp}_{\bf c}(f^{-\nu}(\mathcal L_{f^{\nu}({\bf c})}(T))),\]
\[\mathcal H(T)=\mathrm{Comp}_{\bf c}(f^{-\nu}(T)).\]
Let $T$ be an admissible interval and $T'$ be a $k$-pullback of $T$ for some $k>0$. Consider the chain $\{G_s\}_{s=0}^k$ from $T'$ to $T$, we call $T'$ is a {\em kid} of $T$ if $G_s \cap \mathrm{Crit}(F)=\emptyset $ for $1\le s\le k-1$. \par
For a nice symmetric interval $B$, let $\hat{\mathcal L}_x(B)$ denote the component of $D_{B} \cup B$ containing $x$.
\begin{definition}Given an admissible interval  $T \ni {\bf c}$, by a {successor} of $T$, we mean
an admissible interval  of the form $\hat{\mathcal L}_{\bf c}(\hat T)$, where $\hat T$ is a kid of  $\hat{\mathcal L}_{\bf c'}(T)$ for some ${\bf c'}\in \mathrm{Crit}(F)$.
\end{definition}
Since $F$ is (infinitely) renormalizable, each admissible interval $T$ has a smallest successor and we denote it by $\Gamma(T)$.\par
Then we can define Kozlovski-Shen-vanStrien's enhanced nest (briefly, KSS nest) as following:
let $E_0=I_0$ and for each $k\ge 0$,
\[L_k=\mathcal G(E_k),\]
\[M_{k,0}=\mathcal H(E_k),\]
\[M_{k,j+1}=\Gamma(M_{k,j})~\text{for}~ 0\le j\le 5{\bf N}-1,\]
\[E_{k+1}=M_{k,5{\bf N}}=\Gamma^{5{\bf N}}(\mathcal H(\mathcal G(E_k))).\]
For each $j\ge 0$,  let $r_j$ be the first return time from $\mathcal L_{\bf c_0}(E_j)$ to $E_j$. Since $F$ is renormalizable, there is a smallest nonnegative integer $\chi$ such that $r_{\chi}={\bf N}p$.
Let $m_j$ be the integer such that $F^{m_j}(E_{j+1})\subset E_j$ and $F^{m_j}(\partial E_{j+1})\subset \partial E_j$ for $0\le j\le \chi-1$.
\begin{prop}\label{KSS}There exists $C>0$ such that $E_j$ is $C$-nice for $0\le j\le\chi$,  $m_{j+1}\ge 2m_j$ and $3r_{j+1}\ge m_j$ for $0\le j\le \chi-2$. 
\end{prop}
\begin{proof}See Proposition~8.1 and Lemma~8.3 in \cite{KSS}.
\end{proof}

\begin{Fact}$E_{\chi} \subset I_{m(\kappa)-1}$.
\end{Fact}
\begin{proof}It follows easily from the fact that the return time $r_\chi={\bf N}p$ to $E_{\chi}$ is strictly larger than the return time to $ I_{m(\kappa)-1}$.
\end{proof}

\begin{Coro}\label{co} ${\bf N}p\ge \max (0,\sum\limits_{j=0}^{\chi-5}m_j)$.
\end{Coro}
\begin{proof} By the definition, ${\bf N}p=r_{\chi} > r_{\chi-1}$.  It follows easily Proposition~\ref{KSS} that
\begin{eqnarray*}
 r_{\chi-1}\ge \frac{1}{3}m_{\chi-2}\ge \frac{4}{3} m_{\chi-4}>2m_{\chi-5}\ge m_{\chi-5}+2m_{\chi-6}\ge \cdots\ge \sum\limits_{j=0}^{\chi-5}m_j.
 \end{eqnarray*}
\end{proof}
Now we are ready to prove the Key Lemma.
\begin{proof}[Proof of the Key Lemma]Without loss of generality, we may assume that the renormalization sequences $\mathcal Rf_k$ converges to $g_{\infty}$. For all $k \in \mathbb N$, let $F_k$ be the extended map of $f_k$ and $\Phi_k$ be the orientation-preserving linear map such that $\Phi_k(I_{\infty}(F_k))=[-1,1]$. To simplify the notation, we denote $\kappa_k=\kappa(F_k)$ the height of $F_k$, $I^k_{\infty}=I_{\infty}(F_k)$ for all $k\in \mathbb N$.\par

{\bf Case 1.} If $\sup\limits_{k} |I_{m(\kappa_k)}|/|I^k_{\infty}| \to \infty$, $F^{{\bf N}p_k}_k|_{I_{m(\kappa)+1}}:I_{m(\kappa)+1} \to I_{m(\kappa)}$ can be extended to a polynomial-like map $g_k:U_k\to V_k$ such that $\mathrm{mod} (V_k\setminus J(g_k)) \to \infty$ as $k\to \infty$. Thus $g_{\infty}$ is a polynomial of degree $2^{\bf N}$.\par

{\bf Case 2.}  Assume $\sup\limits_{k} \kappa(F_k)=\infty$ and $\sup\limits_{k} |I_{m(\kappa_k)}|/|I^k_{\infty}| <\Lambda$ for some $\Lambda>0$. In this case, we may assume $\kappa_k \to \infty$. Since $\Phi_k(I^k_{\infty})=[-1,1]$ for all $k$,  we obtain $$2 \le |\Phi_k(I_{m(\kappa_k)})| \le 2\Lambda.$$
Then for each $j\in \mathbb N$,  $\Phi_k (I_{m(\kappa_k-j)})$ are bounded intervals (the bound depends on $j$). In particular, $\Phi_k(E_{\chi(F_k)})$ are bounded intervals since $E_{\chi(F_k)} \subset I_{m(\kappa_k)-1}$, and so are $\Phi_k(E_{\chi(F_k)-5})$.\par
For every $j,k \in \mathbb N$, let $T_{k,j}$ be the component of $F_k^{-{\bf N}p_k}(I_{m(\kappa_k-j)})$ containing $\bf c_0$, by Corollary~\ref{co}, $T_{k,j} \subset E_{\chi(F_k)-5}$. Thus $\sup\limits_{j,k} |\Phi_k(T_{k,j}) |<\infty$.  Passing to a subsequences(use diagonal argument) we may assume $\overline{\Phi_k (I_{m(\kappa_k-j)})}$ and $\overline{\Phi_k(T_{j,k})}$ converge respectively to  closed intervals $D_j$ and $D'_j$ for all $j \in \mathbb N$. Moreover, we have $\sup\limits_{j}|D'_j|<\infty$ since $\sup\limits_{j,k} |\Phi_k(T_{k,j}) |<\infty$.
By Theorem~\ref{realbounds}, $|D_j|\to \infty$, and then $\bigcup D_j=\mathbb R$. It follows from Lemma~\ref{long} that
$g_{\infty}$  has an analytic proper extension $g_{\infty}:D'_j \to D_j$ for every $j\in \mathbb N$. Hence $g_{\infty}$ has a maximal analytical extension to $\bigcup\limits_j D'_j$ which is a bounded interval.\par

{\bf Case 3.}  Now we assume  $\sup\limits_k \kappa(F_k)<\infty$, $\sup\limits_{k} |I_{m(\kappa_k)}|<\Lambda|I^k_{\infty}|$ for some $\Lambda>0$ and $p_k \to \infty$ as $k\to \infty$.  In this case, we have $\sup_k |I_0(F_k)|/|I^k_{\infty}| <\widehat{\Lambda}$ for some $\widehat{\Lambda}>0$.

For any $i\in \mathbb N$, let $T_{k,i}$ be the component of $F^{-i}_k(\mathfrak{E}_0(F_k))$ containing ${\bf c_0}$. Since for any $0\le j\le i$, $F^{{\bf N}p_k}_k(\partial T_{k,{\bf N}p_k-j}) \subset \{\pm 1,\pm \alpha(F_k),\pm \beta(F_k)\}\times\{0\}=:\mathfrak A_k $, there exists $a_k \in \mathfrak A_k$ such that $\#(F^{-{\bf N}p_k}_k(a_k)\cap T_{k,{\bf N}p_k-i}) \ge \frac{i}{8}$. Thus there are at least $\max (\frac{i-8}{8},0)$ critical points of $F^{{\bf N}p_k}_k$ in $T_{k,{\bf N}p_k-i}$.\par

Similar to { Case 2}, we can assume $\Phi_k(T_{k,{\bf N}p_k-i})$ converges to $D_i$ for every $i \in \mathbb N$. Also we can assume $\Phi_k(I_0(F_k))$ converges to a bounded interval $D^{\infty}$, then  $\bigcup\limits_{i} D_i \subset D^{\infty}$.  It follows from  Lemma~\ref{compactness} that $g_{\infty}$ has an analytic extension to $D_i$ for all $i\in \mathbb N$. If $g_{\infty}$ has an analytic extension to some $\Omega \Supset \bigcup\limits_{i} D_i$, then $g_{\infty}$ must be a constant function since $g'_{\infty}$ has infinitely many  zeros in $\bigcup\limits_{i} D_i$.  Hence,  the real trace of $g_{\infty}$ is contained in $D^{\infty}$. \par

\end{proof}

\section{The continuity of the anti-renormalization operator $\mathcal R^{-1}$}In this section, we will recall the definition of renormalizaiton combinatorics introduced by Smania\cite[section~2]{SmaniaPhase} and prove Theorem~C.\par
Let $J$ and $J'$ be two disjoint intervals which are contained in $\bigcup\limits_{j=0}^{{\bf N}-1}\I\times\{j\}$, we say $J\prec J'$ if there exists $j_0$ such that $J,J'\subset \I\times\{j_0\}$ and $J$ lies to the left of $J'$.
\begin{definition}Denote by $\langle A,A^{Crit},\pi,P,m\rangle$ the {\bf combinatorial data} which contains
\begin{itemize}
\item $A=\bigcup\limits_{j=0}^{{\bf N}-1} B_j$ where $B_j$ is a collection  of disjoint intervals contained in $\I\times\{j\}$ with $\# B_j=m$  for $0\le j\le {\bf N}-1$ ;
\item $A^{Crit}=\{J\in A\mid J \ni (0,j)~~\text{for some} ~~0\le j\le {\bf N}-1\}$ and $\#(A^{Crit}\cap B_j)=1$ for all $0\le j\le {\bf N}-1$;
\item $\pi:A\to A$ is a bijection with the following property: if $c\in A^{Crit}$, then 
$a\prec b \prec c$ implies $\pi(a)\prec \pi(b)\prec \pi(c)$ and $c\prec b \prec a$ implies $\pi(a)\prec \pi(b)\prec \pi(c)$;
\item For any $a\in A$ there exists $c\in A^{Crit}$ so that $\pi^j (c) = a$, for some $j \ge 0$;
\item $(0,0)\in P\in A^{Crit}$.
\end{itemize}
\end{definition}

\begin{definition}Two combinatorial data $\sigma=\langle A,A^{Crit},\pi,P,m\rangle$ and $\tilde\sigma=\langle\tilde A,\tilde A^{Crit},\tilde \pi,\tilde P,\tilde m \rangle$ are equivalent if there exists a bijection $\phi:A\to\tilde A$ such that
\begin{itemize}
\item $\phi(A^{Crit})=\tilde A^{Crit}$;
\item for any $x,y \in A$, $x\prec y$ if and only if $\phi(x)\prec \phi(y)$;
\item $\phi\circ\pi=\tilde\pi\circ\phi$;
\item $\phi(P)=\tilde P$.
\end{itemize}
\end{definition}
We use $\mathcal M=\mathcal M(\sigma)$ to denote the equivalence classes of $\sigma$ and let $\Sigma'$ be the set of all the combinatorics.\par
Now we are going to define the product of two combinatorics.

Let $\mathcal M_1$ and $\mathcal M_2$ be two combinatorics, their product $\mathcal M_1*\mathcal M_2$ is definded as following:\par
Assuming $\sigma_{s}=\langle A_s,A^{Crit}_s,\pi_s,P_s,m_s \rangle$ is a representative of $\mathcal M_s$$(s=1,2)$, we define intervals $J^s_i:=\pi^i_s(P_s)$ for $s=1,2$ and $0\le i<m_s{\bf N}$. Choose a family $\{\psi_i\}_{i=0}^{m_1{\bf N}-1}$ of orientation preserving homeomorphisms such that
\begin{enumerate}
\item $\psi_{i}:J^1_i \to \I\times \{i\mod {\bf N}\}$  for $0\le i <m_1{\bf N}$;
\item if $J^1_i \in A^{Crit}_1$, then $\psi_i(0,i\mod {{\bf N}})= (0,i \mod {\bf N})$.
\end{enumerate}
If $i=j\mod {\bf N}$, then we define $I_{i,j}:=\psi^{-1}_i(J^2_j)$. Then we can define a new combinatorial data $\sigma=\langle A,A^{Crit},P,\pi,m_1m_2 \rangle$ such that
\begin{itemize}
\item $A=\{I_{i,j}\mid i=j \mod {{\bf N}}, ~0\le i<m_1{\bf N},~0\le j<m_2{\bf N} \}$;
\item  $A^{Crit}=\{J\in A\mid J \ni (0,j)~~\text{for some} ~~0\le j\le {\bf N}-1\}$;
\item $\pi:A\to A$ such that \[\pi(I_{i,j})=
\begin{cases}I_{i+1,j+1 \mod(m_1{\bf N},m_2{\bf N})}, ~~if~~ I_{i,j}\in A^{Crit} \\ 
I_{i+1,j \mod(m_1{\bf N},m_2{\bf N})},  ~~otherwise
\end{cases}\]
\item $P=I_{0,0} \ni (0,0)$.
\end{itemize} 
Note that $\sigma$ is a combinatorial data and its equivalence class $\mathcal M:=[\sigma]$ does not depend on the choices in the above construction. Finally, $\mathcal M$ is defined to be the product $\mathcal M_1*\mathcal M_2$.\par
A combinatoric $\mathcal M$ is said to be {\it primitive} if it does not have decomposition $\mathcal M=\mathcal M_1*\mathcal M_2$. Let $\Sigma\subset \Sigma'$ be the set of all the primitive combinatorics.
\begin{definition}Let $f$ be a multimodal map of type $\bf N$ and consider an extended
map $F$ induced by a decomposition $(f_0,\cdots,f_{{\bf N}-1})$ of $f$. If $P$ is a maximal periodic interval
for $F$ of period $k$, then we can associate the following combinatorial data $\sigma=\langle A,A^{Crit},\pi,P,k/{\bf N}\rangle$
\begin{itemize}
 \item $A = \{F^i (P): 0 \le i < k\}$;
\item  $A^{Crit} = \{F^i (P): {\bf c} \in F^i (P)~ ~\text{for some critical point} ~~{\bf c}~~ \text{of}~ F \}$;
\item $\pi: A \to A$ is defined by $\pi(F^i (P)) = F^{i+1 \mod {\bf N}}(P)$.
\end{itemize}
If $\mathcal M=[\sigma]$ is primitive, then we say $f$ is renormalizable with renormalization combinatoric $\mathcal M$.
\end{definition}
We say $f$ is a multimodal maps of type $\bf N$ with combinatorics $(\mathcal M_k)_{k\ge 0} \in \Sigma^{\mathbb N}$ if $f$ is infinitely renormalizable and $\mathcal R^k f$ is renormalizable with renormalization combinatoric $\mathcal M_k$ for all $k\ge 0$.\par
Note that for any combinatorics $\underline{\mathcal M}\in \Sigma^{\mathbb N}$ there exists a  real polynomial in $\mathcal I$ with combinatorics $\underline{\mathcal M}$ (see \cite[section~2.1 and section~5.1]{SmaniaPhase}). By Kozlovski-Shen-vanStrien's combinatorial rigidity theorem \cite{realRigidity}, such a real polynomial is unique.
\begin{Lemma}\label{cc}Assume $\underline{\mathcal M^{m}} \to \underline{\mathcal M}$ in $\Sigma^{\mathbb N}$. If for every $m\in \mathbb N$, $f_m$ is an infinitely renormalizable multimodal map of type $\bf N$ with combinatorics $\underline{\mathcal M^{m}}$, then any limit point $g$ of $\{f_m\}$ is infinitely renormalizable with combinatorics $\underline{\mathcal M}$.
\end{Lemma}
\begin{proof}
 Without loss of generality, we may assume $f_m\to g$. Let $J^k_m$ be the restrictive interval of the $k$-th pre-renormalization. As  $\underline{\mathcal M^{m}} \to \underline{\mathcal M}$,  $\mathcal M^m_l=\mathcal M_l$ for all sufficiently large $m$ and $l\le k$. Thus the periods of $J^k_m$ are the same, and so we can assume $J^k_m$ converges to a periodic interval $0\in J^k$ for $g$. For otherwise, $0$ will be a supperattracting periodic point of $g$. It follows that the post-critical set of $g$ is contained in a solenoidal attractor. Thus $g$ only has repelling periodic point. Hence the restrictive intervals and the post-critical set move continuously, which implies $g$ is infinitely renormalizable with combinatorics $\underline{\mathcal M}$.
\end{proof}
\begin{Lemma}\label{continuity}  If a sequence  $\{f_k\}$ of infinitely renormalizable multimodal maps of type $\bf N$ converges to an infinitely renormalizable multimodal map $f$ of type $\bf N$, then $\mathcal R f_k$ converges to $\mathcal R f$.
\end{Lemma}
\begin{proof}
 Let $J=[-a,a]$ be the renormlization interval of $f$, then one of $\{-a,a\}$ is a repelling periodic point of $f$. Without loss of generality, we assume $a$ is a repelling periodic point of $f$ with period $p$. Since $f_k$ converges to $f$, there exists $a_k$ such that $a_k$ is a repelling periodic point of $f_k$ with period $p$ and $a_k$ converges to $a$. It follows easily $f^p_k|_{[-a_k,a_k]}$ is affine conjugate to a multimodal maps of type $\bf N$ by the uniform convergency. Thus the renormalization period $\widehat p_k$ of $f_k$ is at most  $p$ for all $k$ large.  \par
Without loss of generality, we can assume the renormalization period of $f_k$ is $q\le p$ for all $k\in \mathbb N$ and the corresponding restrictive interval $[-\widehat a_k,\widehat a_k]$ converges to $[-\widehat a, \widehat a]$. Since $f_k$ converges to $f$ in $C^3$-topology, $f^q([-\widehat a, \widehat a]) \subset [-\widehat a, \widehat a]$. Clearly $-\widehat a\ne \widehat a$, for otherwise $\widehat a=0$ will be a supperattracting periodic point of $f$, which is impossible as $f$ is infinitely renormalizable. Thus, $[-\widehat a, \widehat a]$ is a periodic interval of $f$. Note that $f^q|_{[-\widehat a, \widehat a]}$ is a pre-renormalization of $f$. Hence, $\mathcal R f_k$ converges to $\mathcal R f$. 

\end{proof}

\begin{TC} For any totally $R$-invariant precompact subset $\mathcal A' \subset \mathcal I$, the restriction $\mathcal R^{-1}|_{\mathcal A'}$ of the anti-renormalization operator $\mathcal R^{-1}$ to $\mathcal A'$ is continuous.
\end{TC}
\begin{proof}For any sequence $\{f_k\} \subset \mathcal A'$  converging to $f \in \mathcal A'$, we prove $\mathcal R^{-1}f_k \to \mathcal R^{-1}f$ as $k \to \infty$. Let ${\bf N}p_k$ be the renormalization period of $f_k$ for all $k\in \mathbb N$.\par
\medskip
{\bf Claim} $\sup\limits_{k} p_k<\infty$. 
\medskip

{\em Proof of the Claim} : Arguing by contradiction,  we may assume $p_k\to \infty$. Then by the Key Lemma, we know that either 
$$f=\lim\limits_{k\to \infty}\mathcal R(\mathcal R^{-1}f_k)$$
 is a polynomial of degree $2^n$ or it has bounded real trace. But both of these two cases are impossible. Indeed, every polynomial of degree $2^n$ cannot be anti-renormalizable due to the degree. On the other hand,  $f$ is infinitely anti-renormalizable, let $g_{-k}=\mathcal R^{-k}f$. By Theorem~\ref{realbounds}, we know $g_{-k}$ can be extended to $[-\lambda^k,\lambda^k]$ for all $k\in \mathbb N$.  Since $f$ is affinely conjugate to an iteration of $g_{-k}$ for all $k\in \mathbb N$, $f$ cannot have bounded real trace.\par

\medskip

Since the  renormalization periods ${\bf N}p_k$ of $\mathcal R^{-1}f_k$ are bounded, then by a similar argument in the proof of Lemma~\ref{cc}, we can conclude all the limit point $g$ of $\{\mathcal R^{-1}f_k\}$ is renormalizable and $\mathcal Rg=\lim\limits_{k_n}\mathcal R(\mathcal R^{-1}f_{k_n})=f$. Thus $\mathcal R^{-1}f_k$ converges to $\mathcal R^{-1}f$.\par

\end{proof}
The following corollary follows immediately from Lemma~\ref{continuity} and Theorem~C.
\begin{Coro}For any totally $R$-invariant precompact subset $\mathcal A' \subset \mathcal I$, the restriction $\mathcal R|_{\mathcal A'}$ of the  renormalization operator $\mathcal R$ to $\mathcal A'$ is a self-homeomorphism.
\end{Coro}

\section{Polynomial-like extension for multimodal maps of type $\bf N$}
In order to prove Theorem~B, we will use the complex method, that is, we extend $f$ to the complex plane and use the tools in complex dynamics. In this section,  we will first recall some definitions and results in polynomial dynamics. Then we prove each hybrid leaf is homeomorphic to a standard model space $\mathcal E_{\bf N}$, which is simply connected.\par
We say $P$ is a {\em polynomial of type $\bf N$} if there exists quadratic polynomials $P_j=a_jz^{2}-a_j-1$ $(j=0,\cdots,{\bf N}-1)$ such that $P=P_{{\bf N}-1}\circ\cdots\circ P_0$. And  $(P_0,\cdots, P_{\bf N-1})$  is called a {\it quadratic decomposition} of $P$. \par
A {\em polynomial-like map} $f:U\to V$ of degree $d$ is a holomorphic proper map of degree $d$ where $U \Subset V $ are quasidisks. The {\em filled Julia set} of  $f$ is $$K(f):=\{z\mid f^n(z) \in U,~~ \forall n\ge 0\}$$ and the boundary of $K(f)$ is call the {\em Julia set} of $f$. The idea of polynomial-like map was first introduced by Douady and Hubbard \cite{DH85}.\par
In this paper, we will consider a special kind of polynomial-like maps which is called polynomial-like map of type $\bf N$. 

\begin{definition}
A polynomial-like map $f$ is called a polynomial-like map of type $\bf N$ if there exists quasidisks $U=U_0\Subset U_1\Subset \cdots \Subset U_{{\bf N}-1}\Subset U_n=V$ and holomorphic branched double covering  $f_j:U_j \to U_{j+1}$  with a unique critical point $z=0$ $(j=0,\cdots,{\bf N}-1)$ such that $f=f_{{\bf N}-1}\circ\cdots\circ f_0:U \to V$ is a  polynomial-like representatives of $f$.
\end{definition}
We will normalize the polynomial-like map of type $\bf N$ so that $0,-1\in U$ and $f(-1)=-1$.\par

Given a polynomial-like map $f$, the corresponding {\em polynomial-like germ} $[f]$ is an equivalence class of polynomial-like maps $\tilde f$ such that $K(\tilde f)=K(f)$ and $\tilde f=f$ near the filled Julia set $K(f)$. The modulus of a polynomial-like germ is defined as:
\[\mod f= \sup\mathrm{mod} (V\setminus \overline U),\]
where the supremum is taken over all polynomial-like representatives $\tilde f:U\to V$ of $f$.\par
In this paper, we will not distinguish the notion of a polynomial-like map and its corresponding polynomial-like germ, and let $\mathcal{C}_{\bf N}$ be the family of all the normalized polynomial-like germs of type $\bf N$ with connected Julia set. For any $\delta>0$,  $\mathcal C_{\bf N}(\delta)$ is used to denote the set of $f\in \mathcal C_{\bf N}$ with $\mod f\ge \delta$. Note that $\mathrm{mod}~f=\infty$ if and only if $f$ is exactly a polynomial.

\subsection{Topology and Complex analytic structure}
Given a Jordan disk, let $\mathcal B_U$ be the Banach space of functions which are holomorphic on $U$ and continuous up to the boundary $\partial U$ and denote $\parallel\cdot\parallel_U$ the $L_\infty$-norm of $\mathcal B_U$.\par
Now we define the topology of $\mathcal C_{\bf N}$ as following. We say {\em $f_k$ converges to $f$ in $\mathcal C_{\bf N}$ } if and only if there exists quasidisk $W\Supset K(f_k)$ such that $f_k$, $k=1,2,\ldots$ and $f$ are well defined on $\overline W$ for all sufficiently large $k$ and $\parallel f_k-f \parallel_W \to 0$ as $k\to \infty$. We say $\mathcal K \subset \mathcal C_{\bf N}$ is closed if for every sequence $\{f_k\} \subset \mathcal K$, the limit points of $\{f_k\}$  also belong to $\mathcal K$.\par 

Two polynomial-like germ $f$ and $g$ are called {\em hybrid equivalent}, if there exists quasiconformal map $h:\C\to \C$ such that $h\circ f=g\circ h$ near $K(f)$ and $\bar\partial h=0$ almost everywhere on $K(f)$.\par
For  ${\bf b} =(b_0,\cdots,b_{{\bf N}-1}) \in \C^{\bf N}$, let $P_{\bf b}:=(b_{{\bf N}-1}z^{2}-b_{{\bf N}-1}-1)\circ\cdots\circ (b_0z^{2}-b_0-1)$, $\tilde{\mathcal H}({\bf b}):=\{f\mid f~ \text{is hybrid equivalent to} ~P_{\bf b}\}$  and let $\mathcal H({\bf b})$ be  the component of $\tilde{\mathcal  H}({\bf b})$ containing $P_{\bf b}$.  Such an $\mathcal H(\bf b)$ is called a {\em hybrid leaf} and it has a natural topology induced from the topology of $\mathcal C_{\bf N}$. A hybrid leaf is called {\em real symmetric} if it contains a real map.  Let $\mathcal H({\bf b}, \epsilon)$ denote the set of all the $f\in \mathcal H(\bf b)$ with $\mathrm{mod} f\ge\epsilon$, then $\mathcal H({\bf b}, \epsilon)$ is a precompact set and any precompact subset $\mathcal K$ of $\mathcal H(\bf b)$ is contained in some $\mathcal H({\bf b}, \epsilon)$ (see \cite[section~5]{mcmullen2016complex}).\par

\begin{Theorem}For every polynomial-like map $f$ of type $\bf N$, there exists a polynomial of type $\bf N$ hybrid equivalent to $f$.
\end{Theorem}
\begin{proof}The proof is based on  quasiconformal surgery. See \cite[Proposition~4.1]{SmaniaPhase}, \cite[Theorem~A]{IK} and also \cite{shishikurasurgery}.

\end{proof}
On the contrary, we have
\begin{Theorem}\label{decomposition}If a polynomial-like map $f:U \to V$ is hybrid equivalent to some polynomial $P$ of type $\bf N$, then $f$ is a polynomial-like map of type $\bf N$.

\end{Theorem}

\begin{proof}
Let $(P_0,\cdots,P_{{\bf N}-1})$ be a quadratic decomposition of $P$ and $h$ be a hybrid conjugacy between $f$ and $P$. Select quasidisks $U_0\Subset U_1\Subset\cdots\Subset U_{\bf N}$ such that $U_0=U$ and $U_{\bf N}=V$.   Let $W_{\bf N}=h(V)$, $W_{{\bf N}-1}=P_{{\bf N}-1}^{-1}(W_{\bf N})$, $W_{{\bf N}-2}=P_{{\bf N}-2}^{-1}(W_{{\bf N}-1})$, $\cdots$, $W_0=P_{1}^{-1}(W_1)=h(U)$. Choose quasiconformal mappings $\varphi_j:U_j\to W_j$ such that $\varphi(U_{j-1})=W_{j-1}$ and $\varphi_j(0)=0$ for all $1\le j<{\bf N}$. Then $h^{-1}\circ P_{{\bf N}-1}\circ \varphi_{\bf N}  :U_{{\bf N}-1}\to U_{\bf N}$ is a  quasiregular map. Thus, we can choose a quasiconformal map $\psi_{{\bf N}-1}:U_{{\bf N}-1}\to U_{{\bf N}-1}$ such that $\psi_{{\bf N}-1}(0)=0$ and 
$$f_{{\bf N}-1}:=h^{-1}\circ P_{{\bf N}-1}\circ \varphi_{\bf N}\circ \psi_{{\bf N}-1}:U_{{\bf N}-1}\to U_{\bf N}$$
 is a holomorphic proper map. Similarly, there exists quasiconformal map $\psi_{{\bf N}-2}:U_{{\bf N}-2}\to U_{{\bf N}-2}$ such that $\psi_{{\bf N}-2}(0)=0$ and
 $$f_{{\bf N}-2}:=\psi^{-1}_{{\bf N}-1}\circ\varphi_{\bf N}^{-1}\circ P_{{\bf N}-2}\circ \varphi_{{\bf N}-1}\circ\psi_{{\bf N}-2}:U_{{\bf N}-2}\to U_{\bf N-1}$$ 
is a holomorphic proper map. By induction, there exist quasiconformal maps $\psi_j:U_j\to U_j$ such that $\psi_j(0)=0$ and 
$$f_{j}:=\psi^{-1}_{j+1}\circ\varphi_{j+2}^{-1}\circ P_{j}\circ \varphi_{j+1}\circ\psi_{j}:U_j \to U_{j+1}$$
 are holomorphic proper maps ($1\le j\le {\bf N}-2$). Set $f_{0}:=\psi_1^{-1}\circ \varphi_2^{-1}\circ P_0\circ h$, then $f_{{\bf N}-1}\circ \cdots\circ f_1\circ f_0=h^{-1}\circ P_{{\bf N}-1}\circ\cdots\circ P_1\circ P_0\circ h=f$. Now we only need to check whether $f_{0}$ is holomorphic. Denote $f_{{\bf N}-1}\circ\cdots\circ f_2\circ f_1$ by $G$, by differentialing we obtain a.e. $z$
 \[0=\frac{\partial f}{\bar{\partial} z}=\frac{\partial G}{\partial w}\frac{\partial f_0}{\bar\partial z}+\frac{\partial G}{\bar\partial w}\frac{\bar{f}}{\bar\partial z}={G'}\frac{\partial f_0}{\bar\partial z}.\]
It follows from Weyl's Lemma that $f_{0}$ is holomorphic and we are done.

\end{proof}

Modifying Lyubich's argument of complex analytic variety in \cite{lyubich1999feigenbaum}and \cite{regularorstochastic}, we can define the complex analytic structure of $\mathcal H(\bf 0)$ as following. Set $\mathcal B^0_U:=\{f\in \mathcal B_U\mid f^{(j)}(0)=0,j=1,2,\cdots,2^{\bf N}-1\}$, then it is also a Banach space under the $L_\infty$-norm. If $f:U \to V$ is a polynomial-like representative of $f\in \mathcal H(\bf 0)$,  it is easy to see $g \in \mathcal B^0_U(f,\epsilon)$ has a polynomial-like restriction on a quasidisk slightly smaller than $U$ for $\epsilon$ sufficiently small (where $\mathcal B^0_U(f,\epsilon)$ is an $\epsilon$-neighborhood of $f$ in $\mathcal B^0_U$.) Thus, we have a natural continuous inclusion $\mathcal J_{U,f,\epsilon}:\mathcal B^0_U(f,\epsilon) \to \mathcal H(\bf 0)$ and $\mathcal B^0_U(f,\epsilon)$ is called {\em a Banach slice of $\mathcal H(\bf 0)$ centered at $f$}. For convenience of notion, we use $\mathcal S_U$ to stand for a Banach slice without specifying $f$ and $\epsilon$. Roughly speaking, $\mathcal J_U:\mathcal S_U \to \mathcal H(\bf 0)$ can be understood as the local chart of $\mathcal H(\bf 0)$ and Lyubich \cite{lyubich1999feigenbaum,
regularorstochastic} proved:

\begin{Lemma}[Lyubich]The family of local charts $\mathcal J_U$ satisfies the following properties:
\begin{enumerate}
\item[P1]Countable base and Compactness. There is a countable family of Banach Slices $\mathcal S_i=\mathcal S_{U_i}$ such that for any $f \in\mathcal H(\bf 0) $,  the Banach Slice $\mathcal S_{U}$ centered at $f$ is compactly contained in some $\mathcal S_i$.  
\item[P2]Lifting of analyticity. If $W\subset U$, then the inclusion map $\mathcal J_{U,W}:\mathcal S_U \to \mathcal B_W$ is complex analytic. Moreover, let $U \Supset V$. Let us consider a locally
bounded map $\phi: \mathcal V \to \mathcal B_V$ defined on a domain $\mathcal V$ in some Banach space.
Assume that the map $\mathcal J_{V,W}\circ \phi: \mathcal V\to \mathcal B_W$ is analytic, then the map 
$\mathcal J_{V,U}\circ \phi: \mathcal V\to \mathcal B_U$ is also analytic.
\item[P3]Density. If $W \subset U$, then $\mathcal B_U$ is dense in $\mathcal B_W$.
\end{enumerate}
\end{Lemma}
A space with properties $P_1-P_3$ is called a complex analytic variety, then $\mathcal H(\bf 0)$ is a complex analytic variety.\par
A map $\phi:\D \to \mathcal H(\bf 0)$ is called  analytic  if for any $z\in \D$, there exists a small disk $\D(z,\delta)$ and a Banach slice $\mathcal S_U$ such that $\phi(\D(z,\delta)) \subset \mathcal S_U$ and the restriction $\phi|_{\D(z,\delta)}:\phi(\D(z,\delta)) \to\mathcal S_U$ is analytic in the Banach sense. Clearly, $\phi$ is an analytic implies that $\phi$ is continuous. 
\subsection{External maps of polynomial-like germs}
A real analytic circle map $g:\T \to \T$ is called expanding if there exists $k \ge 1$ such that $|\mathrm Dg^k(z)|>1$ for all $z \in \T$ where $\mathrm{D}$ denotes the derivative with respect to $z$.\par
Let $\E_{\bf N}$ be the family of real analytic expanding circle covering maps $g:\T \to \T$ of degree $2^{\bf N}$ normalized so that $g(1)=1$. $\E_{\bf N}$ is simply connected (\cite[Lemma~2.1]{Fullhorseshoeunimodal}). Since $g$ is real analytic and expanding, it can be extended to be a holomorphic covering $g:U \to V$ of degree $2^{\bf N}$ where $U \Subset V$  are annular neighborhood of $\T$. Similarly, we can define the modular of expanding circle maps as:
\[\mod g=\sup \mathrm{mod}(V\setminus \overline{U\cup \D}),\]
where the supremum is taken over all extensions $g:U \to V$ of $g$.\par

We will use the Inductive limit topology of $\E_{\bf N}$ (see \cite[Appendix~2]{lyubich1999feigenbaum}) .  In this topology, a sequence $g_k \in \E_{\bf N}$ converges to $g \in \E_{\bf N}$ if there exists a neighborhood $W$ of $\T$ such that all the $g_k$ admit a holomorphic extension to $W$, and $g_k \to g$  uniformly on $W$, i.e. $\sup\limits_{z \in W}|g_k(z)-g(z)| \to 0$.\par
Let $f \in \mathcal{C}_{\bf N}$, consider the B\"ottcher coordinate $\phi_f:\C\setminus K(f) \to \C\setminus \overline{\D}$, then $g=\phi_f\circ f\circ \phi^{-1}_f$ is well defined in a small outer neighborhood of $\T$, by Schwarz reflection principle, $g$ can be extended to a holomorphic expanding map of degree $2^{\bf N}$ in a neighborhood of $\T$. Such a map $g$ is unique up to a rotation conjugation, thus it can be normalized so that $g \in \E_{\bf N}$ and called an external map of $f$. Unfortunately, such an external map $g$ may not be unique. However, we can construct a canonical external map $g\in \E_{\bf N}$ from a polynomial-like germ $f\in \mathcal C_{\bf N}$. Indeed, we prove the following theorem.
\begin{Theorem}\label{externalstructure}For every ${\bf b}\in \C^{\bf N}$, there exists a homeomorphism $\mathcal I_{\bf b}:\E_{\bf N} \to \mathcal H(\bf b)$. Moreover, $\mod \mathcal I_{\bf b}(g)=\mod g $ for all $g \in \E_{\bf N}$.
\end{Theorem}
\begin{proof}Firstly, given $g \in \E_{\bf N}$, we choose a continuous path $\{g_t\}$ connecting $g_0=z^{2^{\bf N}}$ and $g_1=g$. Then we can construct a continuous family of $K$-q.c maps $h_t:\C\setminus{\D} \to \C\setminus{\D}$
with Beltrami differential $\nu_t$ continuously depending on $t$ such that $h_0=\mathrm{id}$ and $h_t\circ g_0=g_t\circ h_t$ near the unit circle $\T$ in the following way:\par
Since $\{g_t\}$ is compact, there exist representatives $g_t:W^1_t\setminus{\D}\to W^2_t\setminus{\D}$ where $W^1_t$, $W^2_t$ are quasidisks with $\mod W^2_t\setminus W^1_t\ge \epsilon_0$ for some $\epsilon_0>0$. By Gr\"otzsch's extremal problem, we know $W^1_t$ contains a Euclid disk $\D(0,r_1)$ with $r_1>1$. As $g_t$ is uniformly expanding near $\T$, we can choose $1<r_0<r_1$ such that $g^{-1}_t(\D(0,r_0)) \Subset \D(0,r_0)$ for all $t$. Let $\gamma^2_t\equiv \partial \D(0,r_0)$ and $\gamma^1_t= g^{-1}_t(\gamma^2_t)$, define $h_t\equiv\mathrm{id}$ outside $\D(0,r_0)$ and we can lift $h_t$ to $\partial \D(0,r^{1/d}_0)$ such that $g_t\circ h_t=g_0$. Then by extension, we can obtain $K$-q.c maps $h_t:A_0\to A_t$ where $A_t$ is the annulus bounded by $\gamma^1_t$ and $\gamma^2_t$. Moreover, $h_t$ depends continuously on $t$. Finally, we can lift $h_t$ to $\C\setminus \D$ by respecting the dynamics so that the Beltrami differentials $\nu_t$ of $h_t$ satisfying $\parallel \nu_t\parallel \le k=\frac{K-1}{K+1}$ for all $t$.\par
Now we are going to construct a continuous path $\{f_t\} \subset \mathcal H(\bf b)$ from the path $\{g_t\}$. Consider the B\"ottcher coordinate $\xi_{\bf b}: \C\setminus\overline{\D}\to \C\setminus K(P_{\bf b})$, we define Beltrami differentials $\mu_t$ on $\C$ such that $\mu_t=(\xi_{\bf b})_*\nu_t$ on $\C\setminus K(P_{\bf b})$ and $\mu_t=0$ on $K(P_{\bf b})$. By the Measurable Riemann Mapping Theorem, we can obtain a continuous path $\{q_t\}$ of $K$-q.c maps such that $\bar\partial q_t=\mu_t\partial q_t$ and $q_t$ fixes $0$, $1$. Then $Q_t:=q_t\circ P_{\bf b}\circ q^{-1}_t$ defines a polynomial-like map of degree $2^{\bf N}$  which is hybrid to $P_{\bf b}$ and $Q_0=P_{\bf b}$. By Theorem~\ref{decomposition}, $Q_t$ is affinely conjugate to a map in $\mathcal H(\bf b)$. Thus we can use a continuous family of affine transformations $\{A_t\}$ to normalize $Q_t$ so that $f_t:=A_t\circ Q_t\circ A^{-1}_t \in \mathcal H(\bf b)$. Let $\phi_t=A_t\circ q_t$, then $f_t=\phi_t\circ P_{\bf b}\circ \phi_t^{-1}$.\par
\begin{Lemma}For every $t\in [0,1]$, $g_t$ is an external map of the polynomial-like map $f_t$.
\end{Lemma}
\begin{proof}Let $\psi_t:\C\setminus \overline{\D}\to \C\setminus K(f_t)$ be the continuous family of Riemann maps normalized so that $\psi_0=\xi_{\bf b}$ and $\tilde g_t=\psi_t^{-1}\circ f_t\circ \psi_t:\T\to \T$ fixes $1$, hence $\tilde g_t$ is an external map of $f_t$.\par
To see the existence of $\psi_t$, we consider another normalized family of B\"ottcher coordinates $\hat\psi_t:\C\setminus \overline{\D}\to \C\setminus K(f_t)$ such that $\hat\psi_t(\infty)=\infty,\hat\psi_t'(\infty)>0$. It is a continuous family due to the semi-upper continuity of $K(f_t)$ and the semi-lower continuity of $J(f_t)$( see \cite[Lemma~4.15]{lyubich1999feigenbaum} for an example). Then $G_t|_{\gamma}=\hat\psi_t^{-1}\circ f_t\circ\hat\psi_t$ is continuous in $t$ where $\gamma$ is a Jordan curve in $\C\setminus\overline{\D}$. By the Schwarz reflection and Maximum Principle, $G_t:\T\to\T$ is a continuous family. Let $z(t)$ be a fixed point of $G_t$ so that $z(t)$ is continuous and $z(0)=1$. Then $\psi_t:=\hat\psi_t\circ e^{i\mathrm{Arg}z(t)}:\C\setminus \overline{\D}\to \C\setminus K(f_t)$ is continuous in $t$ and and $\tilde g_t=\psi_t^{-1}\circ f_t\circ \psi_t:\T\to \T$ fixes $1$.\par
It remains to show that $g_t=\tilde g_t$ for all $t\in [0,1]$. Clearly $\sigma_t:=\psi^{-1}_t\circ\phi_t\circ \xi_{\bf b}:\C\setminus \overline{\D}\to \C\setminus \overline{\D}$ is a quasiconformal map conjugating $g_0$ to $\tilde g_t$ whose Beltrami differential coincides with that of $h_t$.
It follows that $\lambda_t:=\sigma_t\circ h^{-1}_t:\C\setminus \overline{\D} \to \C\setminus \overline{\D}$ is a rotation conjugating $g_t$ to $\tilde g_t$. Since $1$ is the fixed point of $g_t$, $\lambda_t(1)$ must be one of the fixed point of $\tilde g_t$ by the conjugacy.  Let $t_0=\sup\{t\mid \lambda_x(1)=1,~~x\le t\}$, we claim $t_0=1$. Indeed, $\lambda_0=\mathrm{id}$ implies $\lambda_0(1)=1$, thus $0\le t_0\le 1$ exists. By continuity, we obtain $\lambda_{t_0}(1)=1$. Hence $t_0=1$, for otherwise, there exists $\tilde t$ slightly large than $t_0$ such that $\lambda_x(1)=1$ for all $x\le \tilde t$, which contradicts with the definition of $t_0$. We conclude that $\lambda_t=\mathrm{id}$ and hence $g_t=\tilde g_t$ for all $t\in[0,1]$.

\end{proof}
Let $\mathcal I_{\bf b}(g):=f_1$ and we should check that the polynomial-like map $f_1$ we constructed above does not depend on the choice of the path $\{g_t\}$.\par
 To this end, we first show that once the connecting path $g_t$ is chosen, the path $f_t$ does not depend on the choice of $h_t$. Suppose $\tilde h_t$ is another $K$-q.c conjugation from $g_0=z^{2^{\bf N}}$ to $g_t$, let $\tilde f_t$ and $\tilde \phi_t$ be the corresponding polynomial-like map and hybrid conjugation. Then $\eta_t=\tilde \phi_t\circ \phi_t^{-1}$ is a hybrid conjugation from $f_t$ to $\tilde f_t$ near $K(f_t)$. Let $\psi_t:\C\setminus \overline{\D} \to \C\setminus K(f_t)$ and $\tilde \psi_t:\C\setminus \overline{\D} \to \C\setminus K(\tilde f_t)$ be the Riemann mappings in the above construction of the externmal maps respectively.\par
 Let $\tilde \eta_t=\tilde\psi_t^{-1}\circ \eta_t\circ \psi_t$, then $\tilde \eta_t(1)$ is a fixed point of $g_t$ and $g_t\circ \tilde \eta_t=\tilde \eta_t\circ g_t$ on $\T$. Indeed, $\tilde \eta_t$ is a $K$-q.c map defined on $U\setminus \overline{\D}$, then it can be extended to $U\setminus \D$. It is continuous in $t$ and $\tilde \eta_0(1)=1$, and hence $\tilde \eta_t(1)=1$.\par
 
\begin{Lemma}Let $g:\T\to \T$ be an expanding circle map , $h$ is an automorphism of $\T$ such that $h\circ g=g\circ h$ and $h(1)=1$, then $h=\mathrm{id}$.
\end{Lemma} 
\begin{proof}It is trivial when $g=m_d(x)=dx \mod 1$, where $d=2^{\bf N}$. It is well-known any expanding circle map is quasisymmetricly conjugate to $dx \mod 1$ (see \cite[the poof of Lemma~3.8]{lyubich1999feigenbaum}), so the conclusion follows easily.
\end{proof}

\begin{Lemma}[{\cite[Lemma~2.1]{ALM}}]Let $S$ be a hyperbolic Riemann surface with boundary $\gamma$ and $H:S\to S$ be a $K$-q.c map  homotopic to the identity rel the boundary, then $\mathrm{d}_S(x,H(x))\le C(K)$ where $\mathrm{d}_S$ is the hyperbolic distance and $C(K)$ is a constant only depend on $K$.
\end{Lemma}
 
 By the above two Lemmas, we conclude that $\tilde \eta_t=\mathrm{id}$ on $\T$ and 
\begin{eqnarray*}
\mathrm{d}_{\tilde U\setminus K(\tilde f_t)}(\eta_t(x),\tilde\psi_t\circ\psi^{-1}_t(x))
&=& \mathrm{d}_{\tilde\psi_t^{-1}(\tilde U\setminus K(\tilde f_t))}(\tilde\psi_t^{-1}\circ\eta_t(x),\psi^{-1}_t(x))\\
&=& \mathrm{d}_{U\setminus K(f_t)}(\tilde\psi_t^{-1}\circ\eta_t\circ\psi_t(z),z)\\
&=& \mathrm{d}_{U\setminus K(f_t)}(\tilde \eta_t(z),z)\\
&\le& C(K)
\end{eqnarray*}
where $x=\psi_t(z)$ and $\eta_t(x) \to K(\tilde f_t)$, $\tilde \psi_t\circ\psi^{-1}_t(x) \to K(\tilde f_t)$ as $z \to K(f_t)$. Hence $\mathrm{d}_{Euclid}(\eta_t(x),\tilde\psi_t\circ\psi^{-1}_t(x))\to 0$ as $z \to K(f_t)$. By \cite[Lemma~2]{DH85}, we obtain a quasiconformal map:
\[Q_t(x)=
\begin{cases}
\tilde \psi_t\circ\psi^{-1}_t(x), & x\in \C\setminus K(\tilde f_t)\\
\eta_t(x), & x\in K(\tilde f_t)
\end{cases}\] 
It follows from Weyl's lemma that $Q_t$ is conformal, and thus $Q_t$ is affine. But $\tilde f_0=f_0=P_{\bf b}$ and $Q_0=\mathrm{id}$, by the continuity, we conclude that $\tilde f_t=f_t$ for all $t\in [0,1]$.\par
Let us now show that the endpoint $f=f_1$ does not change whether the path $\{g_t\}$ is alternated. Given two paths $\{g_t\}$ and $\{\tilde g_t\}$ connecting $z^d$ and $g$, by the simply connectedness of $\E_{\bf N}$, we can choose a homotopy $g_t^s$ with $g_t^0=g_t$ and $g_t^1=\tilde g_t$. For every $s\in [0,1]$, let $f^s_1$ be the polynomial-like map corresponding to the path $\{g^s_t\}$, then $f^s_1$ are hybrid equivalent and $g$ is the external map of $f^s_1$. By a similar argument as above, we can show that there exists a continuous family of affine transformations $\Lambda_s$ such that $\Lambda_0=\mathrm{id}$ and $\Lambda_s \circ f_1^0\circ\Lambda^{-1}_s=f_1^s$. Hence, by the continuity, $f_1^s=f_1^0$ for all $s\in[0,1]$.\par
The construction of $\mathcal I_{\bf b}(g)$ implies the continuity of $\mathcal I_{\bf b}$.\par
Let us now prove that $\mathcal I_{\bf b}$ is a bijection. For every $f\in \mathcal H(\bf b)$, choose a path $\{f_t\}\subset \mathcal H(\bf b)$ to connect $P_{\bf b}$ and $f$. We will prove that $\{f_t\}$ has a unique lift in $\E_{\bf N}$, and this implies that $\mathcal I_{\bf b}$ is a bijection. Consider a continuous family of conformal mappings $\varphi_t: \C\setminus \overline{\D} \to \C\setminus K(f_t)$ such that $\varphi_t(\infty)=\infty$ and $\varphi'_t(\infty)>0$. Then $\tilde g_t:=\varphi^{-1}_t\circ f_t\circ \varphi_t: \T\to \T$ is an analytic expanding map of degree $2^{\bf N}$ for all $t\in[0,1]$ and $\tilde g_0=z^{2^{\bf N}}$. We can choose a continuous family of conformal maps $A_t:\C\setminus \D \to \C\setminus \D$ so that $g_t:=A_t^{-1}\circ \tilde g_t\circ A_t:\T\to \T$ belongs to $\E_{\bf N}$ and $A_0=\mathrm{id}$. It is easy to check that $\mathcal I_{\bf b}(g_t)=f_t$ by the definition, and so $\{g_t\}$ is a lift of $\{f_t\}$. If $\{f_t\}$ has another lift $\{\hat g_t\}$, then $\hat g_0=z^{2^{\bf N}}$ since $P_{\bf b}$ has a unique preimage $z\mapsto z^{2^{\bf N}}$. Let $\psi_t:\C\setminus \overline{\D} \to \C\setminus K(f_t)$ and $\hat\psi_t:\C\setminus \overline{\D} \to \C\setminus K(f_t)$ be the conformal maps so that $f_t\circ \psi_t=\psi_t\circ  g_t$, $\psi_0=\xi_{\bf b}$ and $f_t\circ \hat\psi_t=\hat\psi_t\circ \hat g_t$, $\hat\psi_0=\xi_{\bf b}$ respectively. Set $\eta_t:=\psi^{-1}_t\circ \hat\psi_t$, we get $\eta_t\circ\hat g_t=g_t\circ \eta_t$ and $\eta_0=\mathrm{id}$. Thus $\eta_t(1)$ is one of the fixed point of $g_t$. But $\eta_0(1)=1$, so by continuity, $\eta_t(1)=1$ for all $t\in[0,1]$. It follows from the Theorem of Boundary Correspondence that $\eta_t$ can only be the identity. Hence $\psi_t=\hat\psi_t$, and this implies $g_t=\hat g_t$ for all $t\in [0,1]$. In other words, $\{g_t\}$ is the unique lift of $\{f_t\}$.\par
For every $\epsilon>0$, consider the restriction $\mathcal I_{\bf b}|_{\E_{\bf N}(\epsilon)}:\E_{\bf N}(\epsilon) \to \mathcal H(\bf b,\epsilon)$ of $\mathcal I_{\bf b}$ to $\E_{\bf N}(\epsilon)$. It is a continuous bijection, so by the compactness of $\E_{\bf N}(\epsilon)$, it is a homeomorphism. If $f_n \to f$ in $\mathcal H(\bf b)$, then $\{f_n\}_{n=1}^{\infty}\cup \{f\}$ is a compact subset of $\mathcal H(\bf b)$, thus it is contained in some $\mathcal H(\bf b,\epsilon_0)$. Hence $g_n:=\mathcal I_{\bf b}^{-1}(f_n) \in \E_{\bf N}(\epsilon_0)$ for all $n\in \mathbb N$. Since $\E_{\bf N}(\epsilon_0)$ is compact, every subsequence of $g_n$ has a limit point, and by the continuity and the bijectivity of $\mathcal I_{\bf b}$, the limit point must be $\mathcal I^{-1}_{\bf b}(f)$. This implies $\mathcal I_{\bf b}^{-1}(f_n)$ converges to $\mathcal I^{-1}_{\bf b}(f)$, thus $\mathcal I^{-1}_{\bf b}$ is continuous.
\end{proof}
By Theorem~\ref{externalstructure}, for every $f\in  \mathcal C_{\bf N}$, there exists a unique $\bf b$ such that $f\in \mathcal H(\bf b)$ and we denote it by $\chi(f)$.

\subsection{Complex renormalization for multimodal maps of type $\bf N$}
Let us now define the complex renormalization for multimodal maps of type $\bf N$.

We say a multimodal map $f$ of type $\bf N$ has {\em a polynomial-like extension} if there exists quasidisks $U_0,U_1,\cdots,U_{\bf N}$ such that $f:U_0\to U_{\bf N}$ is a polynomial-like map, $\I\subset U_j$ and $f_j:U_j\to U_{j+1}$ is holomorphic proper for all $0\le j\le {\bf N}-1$ where $(f_j:\I\to \I)_{j=0}^{{\bf N}-1}$ is a unimodal decomposition of $f$. \par

\begin{definition}
A multimodal map $f$ of type $\bf N$ is called complex renormalizable if it is real renormalizable and both itself and its real renormalization $\mathcal R f$ have polynomial-like extensions. The germ of the polynomial-like extension of  $\mathcal R f$ will be called the complex renormalization of $f$.
\end{definition}
In \cite{C2Density}, Shen proved the complex bounds for all the infinitely renormalizable real analytic box map without critical points of odd order:  
\begin{Theorem}[{\cite[Theorem~3']{C2Density}}]\label{shen1}There exists $\epsilon_0>0$ with the following property. If $\mathcal F$  is a compact  family of  infinitely renormalizable multimodal maps of type $\bf N$, then there exists $K>0$ such that for any $k>K$ and $f\in \mathcal F$,  $\mathcal R^kf$ has an polynomial-like extension $\mathcal R^k f:U\to V$ with
\[\mod V\setminus{U}\ge \epsilon_0.\]
\end{Theorem}
Hence, all the maps in $\mathcal I$ are actually infinitely complex renormalizable (see also \cite[Theorem~3]{C2Density}), so from now on we don't distinguish the terminology of real renormalization and complex renormalization for multimodal maps. We mention here that Clark, Trejo and vanStrien   \cite{clark2017complex} proved the complex bounds for all the infinitely renormalizable real analytic box map recently.\par

 For polynomial-like germs of type $\bf N$, one can still easily define the pre-renormalization just as a  first return map. However,  there is not a canonical way to normalize such a first return map to the normalized form. A usual way to do this, is to use the external marking. Nevertheless, we can define the renormalization for polynomial like germ $f$ which is in a hybrid leaf of a complex renormalizable multimodal map $f_*$ of   type $\bf N$ as following:  Choose a path $\{f_t\}_{t\in [0,1]}$ in this hybrid leaf to connect $f_*$ and $f$, let $h_t$ be the hybrid conjugacy between $f_*=f_0$ and $f_t$, then $f_t=h^{-1}_t\circ f_*\circ h_t$.  Suppose the renormalization period of $f_*$ is $p$, then $f_t^p=h^{-1}_t\circ f^p_*\circ h_t$ restricting to some small region is a pre-renormalization for $f_t$. Finally,  there exists a continuous family $\{\Lambda_t\}$ of affine maps such that $\Lambda_t \circ f^p_t\circ \Lambda^{-1}_t \in \mathcal H(\mathcal R f_*)$ for all $t\in [0,1]$ such that $\Lambda_0 \circ f^p_*\circ \Lambda^{-1}_0=\mathcal Rf_*$ and we define $\mathcal Rf:=\Lambda_1\circ f^p\circ \Lambda^{-1}_1$. As $\mathcal H(f_*)$ is homeomorphic to $\mathcal E_{\bf N}$, it is simply connected, so the definition of $\mathcal Rf$ dose not depend on the choice of the path $\{f_t\}$.
%

\section{Path holomorphic structure on hybrid leaves}
In this section, we will use the method of path holomorphic space developed in \cite{Fullhorseshoeunimodal} by Avila and Lyubich. 
Following Avila-Lyubich \cite{Fullhorseshoeunimodal}, we define the path holomorphic structure on all the real-symmetric hybrid leaves.  Under the path holomorphic structure, the renormalization operator between two hybrid leaves is contracting with respect to the corresponding Carathe\'odary metric. Use Avila and Lyubich's idea of cocycles, one can  transfer the beau bounds (uniform a priori  complex bounds) for real maps to the beau bounds for entire hybrid leaves of real maps. Altogether the contracting property for the renormalization operator and the beau bounds, we show the exponential contraction of the renormalization operator along the real-symmetric leaves. Some proofs in this section are similar to the unimodal case, so we will skip these proofs. For details,  we refer the readers to   section $3-8$ in \cite{Fullhorseshoeunimodal}.
\begin{definition}Let $\mathbb X$ be a topological space, a path holomorphic structure $\bf{\mathrm Hol}(\mathbb X)$ on $\mathbb X$ is a family of continuous paths $\Gamma=\{\gamma\mid \gamma:\D\to \mathbb X~\text{continuous}\}$ such that
\begin{enumerate}
\item $\Gamma$ contains all constant maps;
\item for any $\gamma \in \Gamma$ and holomorphic map $\phi: \D \to \D$, the composition $\gamma\circ \phi$ belongs to $\Gamma$.
\end{enumerate}
A topological space $\mathbb X$ equipped with a path holomorphic structure is called a path holomorphic space.   Every element in $\bf{\mathrm Hol}(\mathbb X)$ will be called a holomorphic path.
\end{definition}
For two path holomorphic space $\mathbb X$ and $\mathbb Y$, we say $\Phi:\mathbb X\to \mathbb Y$ is {\em a path holomorphic map} if $\Phi$ maps each holomorphic path in $\mathbb X$ to a holomorphic path in $\mathbb Y$. We denote by $\bf{\mathrm Hol}(\mathbb X,\mathbb Y)$ the set consisting of all the path holomorphic map from $\mathbb X$ to $\mathbb Y$.

\begin{definition}[Holomorphic path in $\mathcal H(\bf 0)$]
Let $\{f_{\lambda}\}_{\lambda\in \D}$ be a continuous path in $\mathcal H(\bf 0)$, we say $\{f_{\lambda}\}_{\lambda\in \D}$ is a holomorphic path if  there exists  a holomorphic motion $h_{\lambda}(z):\D\times \C \to \C$ such that
\begin{enumerate}
\item $h_{0}=\mathrm{id}$; 
\item $f_{\lambda}\circ h_{\lambda}=h_{0}\circ f_{0}$ on $K(f_{0})$;
\item $\bar\partial h_{\lambda}=0$ a.e on $K(f_{0})$.
\end{enumerate}
\end{definition}

\begin{definition}[Locally holomorphic path]
Let $\{f_{\lambda}\}_{\lambda\in \D}$ be a continuous path in $\mathcal H(\bf 0)$, we say $\{f_{\lambda}\}_{\lambda\in \D}$ is a locally holomorphic path if  for any $\lambda_0\in \D$, there exists a disk $\D(\lambda_0,r)$ and a holomorphic motion $h_{\lambda}(z):\D(\lambda_0,r)\times \C \to \C$ such that
\begin{enumerate}
\item $h_{\lambda_0}=\mathrm{id}$; 
\item $f_{\lambda}\circ h_{\lambda}=h_{\lambda_0}\circ f_{\lambda_0}$ on $K(f_{\lambda_0})$;
\item $\bar\partial h_{\lambda}=0$ a.e on $K(f_{\lambda_0})$.
\end{enumerate}
\end{definition}
Note that $\{f_{\lambda}\}_{\lambda\in \D}$ is a holomorphic path if and only if it is a locally holomorphic path.
 Let $\bf{\mathrm Hol}(\mathcal H(\bf 0))$ be the set of all the holomorphic paths in $\mathcal H(\bf 0)$, then $\bf{\mathrm Hol}(\mathcal H(\bf 0))$ is a path holomorphic structure on $\mathcal H(\bf 0)$.\par
 The following lemma explains the relation between path holomorphic structure and analytic structure in $\mathcal H(\bf 0)$.

\begin{Lemma} \label{holopath}A map $\phi:\D \to \mathcal H(\bf 0)$ is a holomorphic path  in $\mathcal H(\bf 0)$ if and only if it is analytic.
\end{Lemma}
\begin{proof} If $\phi:\D\to \mathcal H(\bf 0)$ is  analytic, then for any $\lambda_0 \in \D$, there exist a sufficiently small round disk $\D(\lambda_0,r)$ and a Banach slice $\mathcal S_U=\mathcal B_U(\hat f,\epsilon)$ such that $\phi(\D(\lambda_0,r)) \subset \mathcal S_U$. As $r$ is sufficiently small, we can assume $\epsilon$ is also sufficiently small so that there exist quasidisks $U_f$ slightly smaller than $U$ and a quasidisk $V$ such that $f:U_f \to V$ is polynomial-like for every $f \in \phi(\D(\lambda_0,r))$. Then we can easily construct an analytic family of quasiconformal maps $\{h_{\lambda}:\overline V\setminus U_{\phi(\lambda_0)} \to \overline V\setminus U_{\phi(\lambda)}\}_{\lambda \in \D(\lambda_0,r)}$ which respects the dynamics on the boundaries. Let $\mu_\lambda$ be the Beltrami differential of $h_\lambda$ and pull $\mu_\lambda$ back by $\phi(\lambda_0)$, then $\mu_\lambda$ can be extended to $V\setminus K(\phi(\lambda_0))$. Finally, define $\mu_\lambda=0$ on $\C\setminus V \cup K(\phi(\lambda_0))$ and by the Measurable Riemann Mapping Theorem, we obtain an analytic family of quasiconformal maps $H_{\lambda}$ such that $H_\lambda \circ \phi(\lambda_0)=\phi(\lambda)\circ H_\lambda$ on $U_{\phi(\lambda_0)}$ and $\bar\partial  H_{\lambda}=0$ a.e. on $K(\phi(\lambda_0))$. It follows that $\phi(\D)$ is a holomorphic path in $\mathcal H(\bf 0)$.(Indeed, it is actually a Beltrami path.)\par
Vice versa, if $\phi:\D \to \mathcal H(\bf 0)$ is holomorphic path, set $f_\lambda=\phi(\lambda)$ for all $\lambda \in \D$, then for any $\lambda_0\in \D$ there exist a disk $\D(\lambda_0,r)\subset \D$ and a holomorphic motion $h_{\lambda}(z):\D(\lambda_0,r)\times \C \to \C$ such that $h_{\lambda_0}=\mathrm{id}$, $f_{\lambda}\circ h_{\lambda}=h_{\lambda_0}\circ f_{\lambda_0}$ on $K(f_{\lambda_0})$ and $\bar\partial h_{\lambda}=0$ a.e on $K(f_{\lambda_0})$. Choose $r$ sufficiently small, we can assume there exist quasidisks $U$ and $V$ such that $f_{\lambda_0}:U\to V$ is polynomial-like and $\bigcup\limits_{\lambda\in \D(\lambda_0,r)} K(f_\lambda) \Subset U$. Since $\phi$ is a continuous map, $\phi(\D(\lambda_0,r))$ is contained in some Banach slice $\mathcal S_U$ once $r$ is small. It remains to prove that for every $z \in U$, $f_\lambda(z)$ is holomorphic in $\lambda$. As $f_{\lambda_0}$ is hybrid conjugate to $z^d$,   $\mathrm{int}(K(f_{\lambda_0}))$ is a quasidisk. Since $f$ is a holomorphic function of $(\lambda,z) \in\D(\lambda_0,r)\times\mathrm{int}(K(f_{\lambda_0}))$, Hartog's Theorem implies that $f_{\lambda}(z)$ is in fact  a holomorphic function of $(\lambda,z)$ through $\D(\lambda_0,r)\times U$.

\end{proof}

\begin{definition}[Beltrami path]
For every ${\bf b} \in \mathcal C$,  a path $\{f_{\lambda}\}_{\lambda\in \D(\lambda_0,r)} \subset \mathcal H(\bf b)$ is called a Beltrami path if there exists
a holomorphic motion $h_{\lambda}: \C\to \C$ over $\D(\lambda_0,r)$, based on $\lambda_0$, that provides a hybrid conjugacy
between $f_{\lambda_0}$ and $f_{\lambda}$.  
\end{definition}
The proof of Lemma~\ref{holopath} has implied the following corollary:
\begin{Coro}A continuous path $\{f_\lambda\}_{\lambda\in \D} \subset \mathcal H(\bf 0)$ is a holomorphic path if and only if it is a Beltrami path.
\end{Coro}
Now we are going to use the homeomorphism $\mathcal I_{\bf b}\circ \mathcal I^{-1}_{\bf 0}$ to define the path holomorphic structure on the hybrid leaf $\mathcal H(\bf b)$ for each $\bf b$.

\begin{definition}[Path holomorphic structure on $\mathcal H(\bf b)$]For every $\bf b \in \mathcal{C}$, a continuous path $\{f_\lambda\}_{\lambda\in \D} \subset \mathcal H(\bf b)$ is a holomorphic path if $\{\mathcal I_{\bf b}\circ \mathcal I^{-1}_{\bf 0}(f_\lambda)\}_{\lambda\in \D}$ is a holomorphic path in $\mathcal H(\bf 0)$. 
\end{definition}
Let $\mathrm h_{\mathbb D}(\cdot,\cdot) $ be the hyperbolic metric on $\D$ and let $\displaystyle d_{\D}(\cdot,\cdot):=\frac{e^{\mathrm h_{\D}}-1}{e^{\mathrm h_{\D}}+1}$. By convexity, $d_{\D}$ is a metric on $\D$.\par
For each hybrid leaf $\mathcal H({\bf b})$, following Avila-Lyubich \cite{Fullhorseshoeunimodal}, we define 
\[d_{\mathcal H({\bf b})}(f_1,f_2)=\sup\limits_{\phi \in {\bf \mathrm{Hol}(\mathcal H({\bf b}),\D)}} d_{\D}(\phi(f_1),\phi(f_2)),\]
for any $f_1,f_2 \in d_{\mathcal H({\bf b})}$.
It  is a well-defined  {\em Carathe\'odory metric} on the path holomorphic space $\mathcal H({\bf b})$. (See \cite[Theorem~4.2 and Lemma~4.1]{Fullhorseshoeunimodal}.)

Since the homeomorphism $\mathcal I_{\bf b}\circ \mathcal I^{-1}_{\bf 0}$ and the renormalization operator $\mathcal R$ map Beltrami paths to Beltrami paths, we obtain:
\begin{Coro}For every $\bf b \in \mathcal{C}$, a continuous path $\{f_\lambda\}_{\lambda\in \D} \subset \mathcal H(\bf b)$ is a holomorphic path if and only if it is  a Beltrami path.
\end{Coro}

\begin{Lemma}For every $\bf b_1, \bf b_2 \in \mathcal{C}$, the renormalization operator $\mathcal R:\mathcal H(\bf b_1) \to \mathcal H(\bf b_2)$ is path holomorphic.
\end{Lemma}
Recall that $\mathcal{I}$ is the set of all the infinitely renormalizable multimodal maps of type $\bf N$. A hybrid leaf $\mathcal H(\bf b)$ is called real-symmetric if it contains a polynomial-like extension for some multimodal maps of type $\bf N$. Let $\widehat {\mathcal{I}}=\bigcup\limits_{f\in \mathcal I} \mathcal H(\chi(f))$, where $\mathcal H(\chi(f))$ is the real-symmetric hybrid leaf containing $f$. A family $\mathcal F\subset \widehat{ \mathcal I}$ is said to have {\em beau bounds} if  there exists $\epsilon_0>0$ such that for any $\delta>0$ there is a moment $n_{\delta}$ so that  $\mathrm{mod}(\mathcal R^n f)\ge \epsilon $ for all $n\ge n_{\delta}$ and any $f \in \mathcal F$ with $\mathrm{mod}(f)\ge \delta$. By Theorem~\ref{shen1}, $\mathcal I$ has beau bounds.
Let us restate the following two theorems in \cite{Fullhorseshoeunimodal} to our situation. For more details, we refer the readers to section $6-8$ in \cite{Fullhorseshoeunimodal}.
\begin{Theorem}\cite[Theorem~6.2]{Fullhorseshoeunimodal} There exists $\epsilon_0>0$ with the following property. For any $\gamma>0$ and $\delta>0$ there exists
$N = N(\gamma,\delta)$ such that for any two maps $f,\tilde f \in \mathcal C_{\bf N}(\delta)\cap \widehat{\mathcal I}$ in the same real-symmetric hybrid leaf we have
\[\mathcal R^{k}f ,\mathcal R^k \tilde f \in \mathcal H_{{\bf b}_k}(\epsilon_0), ~~\text{and}~~ \mathrm{d}_{\mathcal H_{{\bf b}_k}}(\mathcal R^{k}f ,\mathcal R^k \tilde f)<\gamma,~~k\ge N,\] 
where ${{\bf b}_k} =\chi(\mathcal R^k f)$.
\end{Theorem}
\begin{proof} For each real-symmetric hybrid leaf $\mathcal H(\bf b)$, we can associate a cocycle $G=G_{\bf b}$ with values in $\bf{\mathrm Hol}(\mathcal H(\bf 0),\mathcal H(\bf 0))$ as following: for each $\bf b \in \mathcal C$, let $\Psi_{\bf b}:=\mathcal I_{\bf 0}\circ \mathcal I^{-1}_{\bf b}:\mathcal H(\bf b)\to \mathcal H(\bf 0)$ and we define 
\[G^{m,n}(\Psi_{\bf b}(f)):=\Psi_{{\bf b}_{n-m}}(\mathcal R^{n-m}(f)),\]
where ${\bf b}_{n-m}=\chi(\mathcal R^{n-m}(f))$.  Let $\mathcal G$ be the set of all such cocycles which correspond to real-symmetric hybrid
leaves. 
Similar to the unicritical case, one can show that $\mathcal G$ satisfies the hypotheses of \cite[Theorem~6.3]{Fullhorseshoeunimodal} (see \cite[Lemma~6.4,Lemma~6.5]{Fullhorseshoeunimodal} for an example). Then the Theorem follows from~\cite[Theorem~6.3]{Fullhorseshoeunimodal}.
\end{proof}

\begin{Theorem}\label{contraction}\cite[Theorem~5.1]{Fullhorseshoeunimodal}Let $\mathcal F\subset \mathcal C_{\bf N}$ be a family of infinitely renormalizable maps with beau bounds which is forward invariant under renormalization. If $\mathcal F$ is a union of entire hybrid leaves then there exists $\lambda<1$ such that whenever $f, \tilde f \in \mathcal F$ are in the same hybrid leaf, we have
\[\mathrm{d}_{\mathcal H_{{\bf b}_k}}(\mathcal R^k f , \mathcal R^k \tilde f )\le C\lambda^k, k \in \mathbb N,\]
where ${{\bf b}_k} =\chi(\mathcal R^k f)$ and $C > 0$ only depends on $\mathrm{mod} f$  and $\mathrm{mod} \tilde f$.
\end{Theorem}
Combine these two theorem, we conclude that
\begin{Coro}For any two $f,\tilde f \in \widehat{\mathcal I}$ in the same real-symmetric hybrid leaf, then there exists $C>0$ and $0<\lambda<1$ such that for sufficiently large $k\in \mathbb N$,
\[\mathrm{d}_{\mathcal H_{{\bf b}_k}}(\mathcal R^k f, \mathcal R^k \tilde f )\le C\lambda^k.\]
where ${{\bf b}_k} =\chi(\mathcal R^k f)$ and $C > 0$ only depends on $\mathrm{mod} f$  and $\mathrm{mod} \tilde f$.
\end{Coro}
Now we are going to prove Theorem~B.
\begin{TB}Let $\mathcal I$ and $\Sigma$ be as in the assumptions of Theorem~A. Then there exists a precompact subset $\mathcal A\subset \mathcal I$ and a topological semi-conjugacy between $\mathcal R|_{\mathcal A}$ and  a two-sided  full shift  on $\Sigma^{\mathbb Z}$.
\end{TB}
\begin{proof}The proof is similar to Avila and Lyubich's, but for completeness we give a proof here. There exists $\epsilon_0>0$ with the following property due to the beau bounds for $\widehat{\mathcal I}$. For every real-symmetric polynomial $P_{\bf b} \in \widehat{\mathcal I}$, we have $\mathrm{mod} \mathcal R^k P_{\bf b} \ge \epsilon_0$. Given $\underline{\mathcal M}=(\mathcal M_k)_{k\in \mathbb Z}$, for any $k_0<0$, there is a unique real-symmetric polynomial $P_{k_0}$ with combinatorics $(\mathcal M_k)_{k\ge k_0}$. Then for any $l\ge k_0$, set $f_{l,k_0}=\mathcal R^{l-k_0}P_{k_0}$. Clearly, $f_{l,k_0}$ is infinitely renormalizable with combinatorics $(\mathcal M_k)_{k\ge l}$ and $\mathrm{mod} f_{l,k_0}\ge \epsilon_0$. By the precompactness of $\mathcal C_{\bf N}(\epsilon_0)$, we may select a subsequence $k(j)\to -\infty$ such that $f_{l,k(j)}$ converges to some $f_l \in \mathcal C_{\bf N}(\epsilon_0)$. It follows from Lemma~\ref{cc} that $f_l$ is infinitely renormalizable with combinatorics $(\mathcal M_k)_{k\ge l}$. Using the diagonal procedure (going backwards in $l$), we ensure that $\mathcal Rf_{l-1}=f_l$. Then $f_0$ is a bi-infinitely renormalizable map with combinatorics $\underline{\mathcal M}$ and $\mod \mathcal R^k f_0 \ge \epsilon_0$ for all $k\in \mathbb Z$.\par
Let us now prove that the $f_0$ we constructed in the above paragraph is unique. If $\tilde f_0$ is another bi-infinitely renormalizable map with combinatorics $\underline{\mathcal M}$ and $\mathrm{mod} \mathcal R^k \tilde f_0 \ge \epsilon_0$ for all $k\in \mathbb Z$. By the combinatorial rigidity, $\mathcal R^k f_0$ and $\mathcal R^k \tilde f_0$ are in the same real-symmetric hybrid leaf for all $k \in \mathbb Z$. It follows from Theorem~\ref{contraction} that
\[\mathrm{d}_{\mathcal H_{{\bf b}_k}}(\mathcal R^k f_0 , \mathcal R^k \tilde f_0 )=\mathrm{d}_{\mathcal H_{{\bf b}_k}}(\mathcal R^l \mathcal R^{k-1}f_0 , \mathcal R^l\mathcal R^{k-l} \tilde f_0 )\le C\lambda^l, l \in \mathbb N,\]
and let $l\to +\infty$ we get $\mathcal R^k f_0=\mathcal R^k \tilde f_0$ for all $k\in \mathbb Z$. In particular, $f_0=\tilde f_0$.\par
Now we define $h:\Sigma^{\mathbb Z} \to \mathcal F$ such that $h(\underline{\mathcal M})=f_0$ where $f_0$ is a bi-infinitely renormalizable map with combinatorics $\underline{\mathcal M}$. Since the existence and uniqueness of $f_0$ has been proven, the map $h$ is well-defined. Let $\mathcal A:=h(\Sigma^{\mathbb Z})$, then $h:\Sigma^{\mathbb Z} \to \mathcal A$ is surjective.\par
To see the injectivity of $h$, we assume that if there exist $\underline{\mathcal M}^1\ne \underline{\mathcal M}^2$ such that $h(\underline{\mathcal M}^1)=h(\underline{\mathcal M}^2)=f$. Then by the definition of $h$, $f$ is bi-infinitely renormalizable with combinatorics $\underline{\mathcal M}^1$ and $\underline{\mathcal M}^2$. Clearly, $\mathcal M^1_k=\mathcal M^2_k$ for all $k\ge 0$. However, by the injectivity of the renormalization operator $\mathcal R$ (see \cite [Proposition 2.2]{smania2016solenoidal}), $\mathcal R^{k}f$ is a singleton for $k\in \mathbb Z_-$ and then $\mathcal M^1_k=\mathcal M^2_k$ for all $k\in \mathbb Z_-$. Thus $\underline{\mathcal M}^1= \underline{\mathcal M}^2$.\par
Finally, the continuity of $h$ follows easily from Lemma~\ref{cc} and we are done.
\end{proof}
To prove Theorem~A, it remains to prove the following lemma:
\begin{Lemma}If $f_k\to f$ in $\mathcal A$, then $h^{-1}(f_k)\to h^{-1}(f)$ in $\Sigma^{\mathbb Z}$.
\end{Lemma}
\begin{proof}
Let $\underline{\mathcal M}=(\mathcal M_l)_{l\in \mathbb Z}$ and $\underline{\mathcal M}^k=(\mathcal M^k_l)_{l\in \mathbb Z}$ be the combinatorics of $f$ and $f_k$ $(k\in \mathbb N)$ respectively. It sufficies to show that for any $l_0\in \mathbb N$, $\mathcal M^k_l=\mathcal M_l$ for all $-l_0\le l \le l_0$ and  $k$ sufficiently large.\par
Since the periodic points of an infinitely renormalizable map are all repelling, the post-critical set moves continuously in a neighborhood of $f$. Thus it is clear that $\mathcal M^k_l=\mathcal M_l$ for $0\le l\le l_0$
 and $k$ large.\par
 To prove  $\mathcal M^k_l=\mathcal M_l$ for all $-l_0\le l <0$ and $k$ sufficiently large, we just need to prove $\mathcal R^{-l}f_k \to \mathcal R^{-l}f$ for $-l_0\le l<0 $.  By Theorem~C, $\mathcal R^{-l}$ is continuous and we are done.
\end{proof}

\bibliographystyle{abbrv}
\bibliography{multimodal renormalization}
\end{document}